\documentclass{compositio}

\usepackage[T1]{fontenc}        
\usepackage[latin1]{inputenc}   
\usepackage[francais, english]{babel}
\usepackage{lmodern}

\usepackage{amssymb,amsfonts,amsthm,amsmath,latexsym}
\usepackage{enumerate, caption, array}
\usepackage[unicode]{hyperref}

\theoremstyle{plain}
\newtheorem{theoreme}{Théorème}
\newtheorem{coro}{Corollaire}

\newtheorem{lemme}{Lemme}
\newtheorem{prop}{Proposition}

\newtheorem{thmext}{Théorème}

\theoremstyle{definition}

\theoremstyle{remark}
\newtheorem*{remarque}{Remarque}

\newcommand{\cond}{{\rm cond}}
\newcommand{\prim}{\text{ primitif}}

\renewcommand{\bar}{\overline}
\renewcommand{\setminus}{\smallsetminus}
\newcommand{\summ}{\mathop{\sum\sum}}
\newcommand{\summm}{\mathop{\sum\sum\sum}}
\newcommand{\summmm}{\mathop{\sum\sum\sum\sum}}
\newcommand{\e}{{\rm e}}
\newcommand{\dd}{{\rm d}}
\newcommand{\cB}{{\mathcal B}}
\newcommand{\cE}{{\mathcal E}}

\newcommand{\cC}{{\mathcal C}}

\renewcommand{\hat}{\widehat}

\newcommand{\cD}{{\mathcal D}}

\newcommand{\cI}{{\mathcal I}}

\newcommand{\cS}{{\mathcal S}}

\newcommand{\cK}{{\mathcal K}}
\newcommand{\cX}{{\mathcal X}}
\newcommand{\cR}{{\mathcal R}}
\newcommand{\cZ}{{\mathcal Z}}

\newcommand{\ee}{{\varepsilon}}

\newcommand{\bfH}{{\mathbf H}}
\newcommand{\bfN}{{\mathbf N}}
\newcommand{\bfZ}{{\mathbf Z}}

\newcommand{\bfR}{{\mathbf R}}
\newcommand{\bfC}{{\mathbf C}}
\newcommand{\bfUn}{{\mathbf 1}}

\newcommand{\vth}{{\vartheta}}
\newcommand{\vphi}{{\varphi}}
\renewcommand{\tilde}{\widetilde}

\newcommand{\card}{{\rm card\ }}

\newcommand{\floor}[1]{{\left\lfloor {#1} \right\rfloor}}

\renewcommand{\mod}[1]{\ ({\rm mod\ }#1)}

\renewcommand\Re{\operatorname{\mathfrak{Re}}}

\catcode`_=\active
\newcommand_[1]{\ensuremath{\sb{\substack{#1}}}}

\numberwithin{equation}{section}

\title[Théorèmes de type Fouvry--Iwaniec pour les entiers friables]{Théorèmes de type Fouvry--Iwaniec \\ pour les entiers friables}

\keywords{friable integers, equidistribution in arithmetic progressions, dispersion method}
\classification{11N25 (Primary), 11N69 (Secondary)}

\author{Sary Drappeau}
\email{drappeaus@dms.umontreal.ca}
\address{CRM -- Université de Montréal, Pavillon André Aisenstadt, 2920 ch. de la Tour, Montréal, H3T~1J4 QC, Canada}

\date{\today}

\begin{abstract}
An integer~$n$ is said to be~$y$-friable if its largest prime factor~$P^+(n)$ is less than~$y$. In this paper, it is shown that the $y$-friable integers less than $x$ have a weak exponent of distribution at least~$3/5-\ee$ when $(\log x)^c\leq x\leq x^{1/c}$ for some $c=c(\ee)\geq 1$, that is to say, they are well distributed in the residue classes of a fixed integer~$a$, on average over moduli~$\leq x^{3/5-\varepsilon}$ for each fixed~$a\neq 0$ and~$\varepsilon>0$. We apply this to the estimation of
the sum $\sum_{2\leq n\leq x, P^+(n)\leq y}\tau(n-1)$ when $(\log x)^c\leq y$. This follows and improves on previous work of Fouvry and Tenenbaum. Our proof combines the dispersion method of  Linnik in the setting of Bombieri, Fouvry, Friedlander and Iwaniec, and recent work of Harper on friable integers in arithmetic progressions.
\end{abstract}

\begin{document}

\maketitle
\tableofcontents

\section{Introduction}

Un entier~$n$ est dit~\emph{$y$-friable} si son plus grand facteur premier~$P(n)$ est inférieur ou égal à~$y$, avec la convention~$P(1)=1$.
L'ensemble des entiers inférieurs ou égaux à~$x$ qui sont~$y$-friables est noté~$S(x, y)$, et on pose~$\Psi(x, y) = \card S(x, y)$.
L'étude de~$S(x, y)$ est l'objet d'abondantes études, les outils variant selon la taille de~$y$ par rapport à~$x$ : on réfère le lecteur à l'article de survol~\cite{TH93}. On s'intéresse ici à la répartition de~$S(x, y)$ dans les progressions arithmétiques. On définit
\[ \Psi(x, y ; a, q) := \card\{n \in S(x, y)\ |\ n\equiv a \mod{q} \},\]
\[ \Psi_q(x, y) := \card \{n\in S(x,y)\ |\ (n,q)=1\} .\]
Harper~\cite{HarperKS12}, précisant un résultat de Soundararajan~\cite{Sound08}, montre que pour tout~$\ee>0$ fixé, lorsque~$(a, q)=1$,
\begin{equation}\label{estim-psi-ap-single}
\Psi(x,y;a,q) \sim \frac{\Psi_q(x, y)}{\vphi(q)} \quad (\log x/\log q\to\infty, q\leq y^{4\sqrt{e}-\ee}, y\geq y_0(\ee)).
\end{equation}
Soundararajan conjecture que cette estimation a lieu pour tout~$A>0$ fixé lorsque~$q\leq y^A$.
Comme il est noté dans~\cite{Granville93b} et~\cite{Sound08}, établir cette estimation pour un~$A>4\sqrt{e}$
aurait des conséquences intéressantes sur le problème du plus petit résidu non quadratique modulo~$p$,
au sujet duquel on réfère le lecteur aux travaux de Burgess~\cite{Burgess}.

La situation présente des similarités avec la suite des entiers premiers : s'il est délicat de majorer le terme
\[ E(x, y ; a, q) := \Psi(x, y ; a, q) - \frac{1}{\vphi(q)} \Psi_q(x, y) \]
uniformément dans un large domaine en~$x$, $y$ et~$q$, on peut espérer obtenir des résultats en moyenne, du même type que le résultat
de Bombieri--Vinogradov, c'est-à-dire des majorations de la somme
\[ \sum_{q\leq Q}\max_{(a, q)=1}|E(x, y; a, q)| .\]
Dans la plupart des applications, et plus particulièrement celles liées aux techniques de crible, ce sont ces objets qui interviennent.
La question de l'uniformité en~$Q$ est cruciale.

On définit les quantités suivantes, qui interviennent fréquemment dans l'étude des entiers friables :
\[ u := \frac{\log x}{\log y}, \qquad H(u) := \exp\{u/(\log (u+1))^2\} .\]
Harper~\cite{HarperBV2012}, précisant des résultats de Fouvry--Tenenbaum~\cite{FT96}, Granville~\cite{AG93} et Wolke~\cite{Wolke},
montre qu'il existe~$c, \delta>0$ tels que pour tout~$A\geq0$, lorsque~$(\log x)^c\leq y\leq x$ et~$Q\leq \sqrt{\Psi(x, y)}$, on ait
\begin{equation}\label{majo-harper}
\sum_{q\leq Q}\max_{(a, q)=1}|E(x, y; a, q)| \ll_A \Psi(x, y)\big\{H(u)^{-\delta}(\log x)^{-A}+y^{-\delta}\big\} + Q\sqrt{\Psi(x, y)}(\log x)^{7/2}
.\end{equation}
La constante implicite est effective si~$A<1$\footnote{Cela découle par exemple de la proposition~II.8.30 de~\cite{Tene2007}}.
Dans les applications, il est important que le majorant soit~$o(\Psi(x, y))$,
et le résultat de Harper assure cela dès que~$Q=o(\sqrt{\Psi(x, y)}(\log x)^{-7/2})$. De même que dans la situation
du théorème de Bombieri--Vinogradov, il serait très intéressant d'avoir un majorant qui soit~$o(\Psi(x, y))$
lorsque~$Q$ est de l'ordre de~$\sqrt{x}$.

Dans le contexte des nombres premiers et du théorème de Bombieri--Vinogradov, de tels résultats peuvent être obtenus
si l'on fixe l'entier~$a$ dont on détecte la classe de congruence. Cela trouve son origine dans des travaux de Fouvry--Iwaniec
et Fouvry~\cite{FI80, Fouvry82, Fouvry84} et est étudié par Bombieri--Friedlander--Iwaniec dans une série d'articles~\cite{BFI,BFI2,BFI3}.
Il est par exemple montré dans~\cite[main theorem]{BFI2} que pour tout~$a\neq 0$ et~$A>0$, il existe~$B=B(A)$ tel que l'on ait
\[ \sum_{q\leq \sqrt{x}(\log x)^A\\(q, a)=1}\Big|\pi(x; a, q) - \frac{{\rm li} x}{\vphi(q)}\Big|\ll_{a, A} x\frac{(\log\log x)^B}{(\log x)^2} \]
où~$\pi(x; a, q)$ désigne le nombre de nombres premiers inférieurs à~$x$ et congrus à~$a$ modulo~$q$.
On pourra consulter l'article de Fiorilli~\cite{Fiorilli} pour plus de références et des résultats récents à ce sujet.

Fouvry et Tenenbaum~\cite[théorèmes~2 et~3]{FT96} montrent par des méthodes similaires le résultat suivant.
\begin{thmext}\label{thm-bfi-ft}
Pour tout~$\ee>0$, il existe~$\delta>0$ tel que lorsque~$a\in\bfZ\setminus\{0\}$ et~$1\leq y\leq x^\delta$, pour tout~$A>0$, on ait
\begin{equation}\label{estim-bfi-ft} \sum_{q\leq x^{3/5-\ee}\\(q, a)=1}|E(x, y; a, q)| \ll_A \frac x{(\log x)^A} \qquad (|a|\leq x^\delta),\end{equation}
\[ \sum_{q\leq x^{6/11-\ee}\\(q, a)=1}|E(x, y; a, q)| \ll_A \frac x{(\log x)^A} \qquad (|a|\leq x). \]
\end{thmext}
Il est intéressant dans ce résultat que~$q$ est autorisé à prendre des valeurs de l'ordre de~$\sqrt{x}$. En revanche, le majorant
n'est~$O(\Psi(x, y))$ que dans un domaine du type
\[ \exp\Big\{(\eta+o(1))\frac{\log x\log\log\log x}{\log\log x}\Big\} \leq y \]
lorsque~$x\to\infty$ et pour un certain~$\eta = \eta(A)$, ce qui restreint l'uniformité en~$x$ et~$y$ dans les applications.
On montre ici le résultat suivant.
\begin{theoreme}\label{thm-BFI-Sxy}
Pour tout~$\ee>0$ fixé, il existe~$c, \delta>0$ pouvant dépendre de~$\ee$ tels que lorsque~$x, y$ vérifient~$(\log x)^c \leq y \leq x^{1/c}$ pour tout~$A\geq0$, on ait
\begin{equation}\label{estim-Sxy-BV}
\sum_{q \leq x^{3/5-\ee}\\(q, a_1a_2)=1} |E(x, y ; a_1\bar{a_2}, q)| \ll_A \Psi(x, y)\big\{H(u)^{-\delta}(\log x)^{-A} + y^{-\delta}\big\}
\qquad (|a_1|, |a_2|\leq x^\delta),
\end{equation}
\begin{equation}\label{estim-Sxy-BV-unif}
\sum_{q \leq x^{6/11-\ee}\\(q, a_1a_2)=1} |E(x, y ; a_1\bar{a_2}, q)| \ll_A \Psi(x, y)\big\{H(u)^{-\delta}(\log x)^{-A} + y^{-\delta}\big\}
\qquad (|a_1|\leq x^{1-\ee}, |a_2|\leq x^\delta)
\end{equation}
lorsque~$a_1, a_2\in\bfZ\smallsetminus\{0\}, (a_1, a_2)=1$, et où la constante implicite est effective si~$A<1$. Dans les membres de gauche de ces estimations,~$\bar{a_2}$ désigne l'inverse de~$a_2$ modulo~$q$.
\end{theoreme}
Les majorants dans~\eqref{estim-Sxy-BV} et~\eqref{estim-Sxy-BV-unif} sont toujours~$o(\Psi(x, y))$ lorsque~$x, y\to\infty$ et~$A>0$
dans leur domaine de validité. Par rapport aux précédents exemples de suites ayant un exposant de répartition supérieur à~$1/2$ pour
une classe de congruence fixée, la suite d'entiers~$S(x, (\log x)^c)$ est particulièrement peu dense puisque~$\Psi(x, (\log x)^c)=x^{1-1/c+o(1)}$
lorsque~$c\geq 1$ est fixé et~$x\to\infty$ (voir la formule~\eqref{HT-psi} {\it infra}).

Par un argument de découpage dichotomique similaire aux calculs de la section~3.1 de~\cite{FT96}, il découle la version plus forte suivante :
\begin{coro}\label{coro-max}
Pour tout~$\ee>0$, il existe~$c, \delta>0$ pouvant dépendre de~$\ee$ tels que lorsque~$x, y$ vérifient~$(\log x)^c\leq y\leq x^{1/c}$, pour tout~$A\geq -1$ fixé, on ait
\[
\sum_{q \leq x^{3/5-\ee}\\(q, a_1a_2)=1} \max_{z\leq x}|E(z, y ; a_1\bar{a_2}, q)| \ll_A \Psi(x, y)\big\{H(u)^{-\delta}(\log x)^{-A} + y^{-\delta}\big\}
\qquad (|a_1|, |a_2|\leq x^\delta),
\]
\[
\sum_{q \leq x^{6/11-\ee}\\(q, a_1a_2)=1} \max_{z\leq x}|E(z, y ; a_1\bar{a_2}, q)| \ll_A \Psi(x, y)\big\{H(u)^{-\delta}(\log x)^{-A} + y^{-\delta}\big\}
\qquad (|a_1|\leq x^{1-\ee}, |a_2|\leq x^\delta),
\]
la constante implicite étant ineffective si~$A>-1$.
\end{coro}
Une analyse plus fine des calculs de Harper menant au Lemme~\ref{lemme-harper} permet de considérer plus généralement la somme
\[ \sum_{q \leq x^{3/5-\ee}\\(q, a_1a_2)=1} \lambda(q)\max_{z\leq x}|E(z, y ; a_1\bar{a_2}, q)| \]
où le poids~$\lambda(q)\geq 0$ vérifie~$\lambda(q)=q^{o(1)}$ lorsque~$q\to\infty$, ainsi que des hypothèses de majoration, notamment sur sa valeur moyenne. Cela n'est pas étudié ici.

Ces résultats permettent par exemple d'estimer la somme
\[ T(x, y) := \sum_{n\in S(x, y)\\n>1}\tau(n-1) .\]
Ce problème peut être vu comme un analogue friable du problème des diviseurs de Titchmarsh ({\it cf.} le corollary~1 de~\cite{BFI}
et le corollaire~2 de~\cite{Fouvry85}). Il est étudié par Fouvry et Tenenbaum dans~\cite{FT90}, qui obtiennent en particulier le résultat suivant.
Dans toute la suite, on désigne par~$\log_k x$ la $k$-ième itérée du logarithme évalué en~$x$.
\begin{thmext}\label{tit-ft}
Soit~$\eta>0$. Lorsque~$x$ et~$y$ sont suffisamment grands et vérifient l'inégalité
\[ \exp\{\eta\log x\log_3 x/\log_2 x\}\leq y\leq x, \]
on a
\begin{equation*} T(x, y) = \Psi(x, y)\log x\Big\{1 + O\Big(\frac{\log(u+1)}{\log y}\Big)\Big\} .\end{equation*}
\end{thmext}
L'énoncé de notre résultat fait intervenir une notation de~\cite{HT84}. Lorsque~$2\leq y\leq x$, on définit implicitement
le \emph{point-selle} $\alpha(x, y)\in[0, 1]$ par l'équation
\begin{equation}\label{def-alpha} \sum_{p\leq y}\frac{\log p}{p^\alpha-1} = \log x .\end{equation}
Lorsque~$y, u\to\infty$ avec~$(\log x)^{2}\leq y\leq x$, on a~$1-\alpha\sim (\log u)/\log y$, et lorsque~$x, y\to\infty$, on~a
\begin{equation}\label{HT-psi} \Psi(x, y) = x^{\alpha + o(1)} .\end{equation}
\begin{theoreme}\label{tit-bfi}
Il existe~$c>0$ tel que lorsque~$(\log x)^c\leq y\leq x$, on ait
\begin{equation}\label{estim-tit-ft}
T(x, y) = C(\alpha)\Psi(x, y)\log x\Big\{1 + O\Big(\min\Big\{\frac1u,\frac{\log(u+1)}{\log y}\Big\}\Big)\Big\}
\end{equation}
avec
\[ C(\alpha) := \prod_p\Big(1-\frac{p^{-\alpha}-p^{-1}}{p-1}\Big) .\]
\end{theoreme}
\begin{remarque}
Le terme d'erreur~$O(1/u)$ est typique de l'utilisation de la méthode du col dans l'étude des entiers friables.
On a~$C(\alpha)\asymp 1$, ainsi que
\[ C(\alpha) = 1 + O\Big(\frac{\log(u+1)}{\log y}\Big) ,\]
et pour tout~$c>0$, on a~$1/u\gg\log(u+1)/\log y$ lorsque~$\log y\gg\sqrt{\log x\log_2x}$.

La constante implicite dans l'estimation~\eqref{estim-tit-ft} obtenue par la méthode proposée ici est effective
lorsque~$H(u)^\delta\gg \log y/\log(u+1)$, ce qui est réalisé par exemple lorsque~$\log y\ll \log x/(\log_2x(\log_3x)^2)$.

La dépendance du terme principal en~$\alpha$ est due à un biais dans la répartition des entiers friables dans les progressions arithmétiques,
biais d'autant plus important que~$\alpha$ (donc~$y$) est petit. Par exemple, un entier~$y$-friable choisi au hasard entre~$1$ et~$x$
est pair avec une probabilité~$2^{-\alpha}+o(1)$ lorsque~$x\to\infty$, {\it cf.} le Lemme~\ref{bt05} {\it infra}.

De même que dans~\cite{FT90}, la preuve que l'on présente ici permet l'obtention d'un développement de~$T(x, y)$
selon les puissances négatives de~$\log y$. On définit le domaine
\begin{equation}\tag{$H_\ee$} 3\leq\exp\{(\log_2x)^{5/3+\ee}\}\leq y\leq x .\end{equation}
Pour tout~$k\geq 0$ fixé, il existe des fonctions~$(u\mapsto\sigma_i(u))_{i\geq 0}$,
définies sur~$[1,\infty[$ par la formule~(1.12) de~\cite{FT90} (en particulier~$\sigma_0(u)=\rho(u)$ la fonction de Dickman, définie en~\eqref{def-dickman} {\it infra}), telles qu'en posant pour tout~$z>1$,
\[ \cC(z) = \cC_k(z) := \bigcup_{j=0}^k[j+\min\{1,(k+1-j)(\log z)/z\}, j+1]\cup[k+1, \infty[ ,\]
on ait pour~$(x, y)\in(H_\ee)$ et~$u\in\cC((\log y)/2)$ l'estimation
\[ T(x, y) = x\log x\Big\{\sum_{j=0}^k\frac{\sigma_j(u)}{(\log y)^j}+O_{\ee,k}\Big(\frac{|\rho^{(k+1)}(u)|}{(\log y)^{k+1}}\Big) \Big\}.\]
\end{remarque}

\begin{acknowledgements}
Ce travail a été réalisé dans la cadre de la thèse de doctorat de l'auteur à l'Université Paris Diderot -- Paris 7. Il tient à exprimer sa plus profonde reconnaissance à son directeur de thèse Régis de la Bretèche et à Étienne Fouvry pour leurs conseils et leurs encouragements durant la rédaction de ce travail, ainsi que son extrême gratitude au rapporteur anonyme pour sa relecture attentive et des remarques ayant permis d'améliorer ce manuscrit.
\end{acknowledgements}

\section{Lemmes}

\subsection{Entiers friables}

Dans~\cite{Hild86}, Hildebrand étudie la fonction~$\Psi(x, y)$ en itérant une équation fonctionelle. On rappelle que la quantité~$u$ est liée à~$x$ et~$y$ par la relation~$u=(\log x)/\log y$, et on définit la fonction de Dickman~$\rho(u)$ pour tout~$u\geq 0$ par
\begin{equation}\label{def-dickman} \rho'(u)+u\rho(u-1)=0, \quad (u\geq 1) \end{equation}
avec la condition initiale~$\rho(u)=1$ pour~$u\in[0,1]$. On prolonge~$\rho$ à~$\bfR$ par~$\rho(u)=0$ pour~$u<0$.
Lorsque~$(x, y)\in(H_\ee)$, on a
\begin{equation}\label{sim-psi-hild}
\Psi(x, y) = x\rho(u)\Big\{1+O_\ee\Big(\frac{\log(u+1)}{\log y}\Big)\Big\}
.\end{equation}
Le terme d'erreur est optimal pour le terme principal~$x\rho(u)$. Saias~\cite{Saias} montre que l'on peut préciser cette estimation
au prix d'un terme principal faisant intervenir la fonction~$\rho$ de façon plus fine. On définit pour~$2\leq y\leq x$,
\[ \Lambda(x, y) := \begin{cases}x\int_{0-}^{+\infty}\rho(u-v)\dd(\floor{y^v}/y^v)&\text{ si }x\not\in\bfZ\\
\Lambda(x+0,y)&\text{ sinon,}\end{cases} \]
ainsi que~$L_\ee(y) := \exp\{(\log y)^{3/5-\ee}\}$. Lorsque~$(x,y)\in H_\ee$, on a
\[ \Psi(x, y) = \Lambda(x, y)\big\{1+O(L_\ee(y)^{-1})\big\} .\]
Le domaine de validité de l'estimation~\eqref{sim-psi-hild} a un lien direct avec le terme d'erreur dans le théorème des nombres premiers.
Ainsi, la forme du domaine~$(H_\ee)$ est liée à la région sans zéro de Vinogradov--Korobov~\cite[formule~(II.3.64)]{Tene2007}
tandis que Hildebrand~\cite{Hild84} montre que la validité de l'estimation
\[ \Psi(x, y) = x\rho(u)O_\ee(y^\ee), \qquad ((\log x)^{2+\ee}\leq y) \]
pour tout~$\ee>0$ est équivalente à l'hypothèse de Riemann.

Si l'on concède de travailler avec un terme principal dépendant moins explicitement de~$x$ et~$y$, il est possible d'obtenir une
estimation valable dans un domaine beaucoup plus étendu. On pose pour~$s\in\bfC, \Re s>0$,
\[ \zeta(s, y) := \prod_{p\leq y}(1-p^{-s})^{-1}, \quad \phi_2(s, y) := \sum_{p\leq y}\frac{p^s(\log p)^2}{(p^s-1)^2}. \]
Alors Hildebrand et Tenenbaum~\cite[theorem~1]{HT84} montrent que lorsque~$2\leq y\leq x$, on a
\begin{equation}\label{sim-psi-ht}
\Psi(x, y) = \frac{x^\alpha\zeta(\alpha, y)}{\alpha\sqrt{2\pi\phi_2(\alpha, y)}}\Big(1+O\Big(\frac1u+\frac{\log y}y\Big)\Big)
.\end{equation}
La quantité~$\phi_2(\alpha, y)$ vérifie lorsque~$2\leq y\leq x$ l'estimation ({\it cf.} \cite[theorem~2]{HT84})
\[ \phi_2(\alpha, y)=\Big(1+\frac{\log x}y\Big)\log x\log y\Big(1+O\Big(\frac1{\log(u+1)}+\frac1{\log y}\Big)\Big) .\]
La démonstration de Hildebrand et Tenenbaum de l'estimation~\eqref{sim-psi-ht} emploie la méthode du col. Celle-ci offre la possibilité
d'obtenir des résultats ``locaux'', permettant par exemple l'évaluation du rapport~$\Psi(x/d, y)/\Psi(x, y)$
dans des domaines en~$x, y$ et~$d$ où aucune approximation lisse de~$\Psi(x, y)$ n'est connue voire possible. Cela est étudié par
La~Bretèche et Tenenbaum dans~\cite{BT05}. Afin de citer leurs résultats on introduit quelques notations supplémentaires.
On définit
\[ g_m(\alpha) := \prod_{p|m}(1-p^{-\alpha}) \qquad (m\in\bfN)\]
et, par analogie avec la fonction~$\Lambda(x, y)$,
\[ \Lambda_m(x, y) := x\int_{0-}^{+\infty}\rho(u-v)\dd R_m(y^v) \]
\[ \text{où }\quad R_m(t) := \frac1t\Bigg(\sum_{n\leq t, (n, m)=1}1\Bigg) - \frac{\vphi(m)}m. \]
Le résultat suivant découle des théorème~2.4, 2.1 et de la formule~(4.1) de~\cite{BT05}. On y utilise la notation~$\omega(m)$ pour désigner le nombre de facteurs premiers distincts de~$m\in\bfN$.
\begin{lemme}\label{bt05}
\begin{enumerate}[(i)]
\item Lorsque~$2\leq y\leq x$ et~$1\leq d\leq x$, on a
\[ \Psi\Big(\frac x d, y\Big) \ll \frac1{d^\alpha}\Psi(x, y) .\]
\item Lorsque~$(\log x)^2\leq y\leq x$, $P^+(m)\leq y$ et~$\omega(m)\ll \sqrt{y}$, on a
\[ \Psi_m(x, y) = g_m(\alpha)\Psi(x, y)\Big\{1+O\Big(\frac{E_m(1+E_m)}{u}\Big)\Big\},\]
où, en notant~$\gamma_m := \log(\omega(m)+1)\log(u+1)/\log y$, la quantité~$E_m$ est définie par
\[ E_m:=(\log(u+1))^{-1}\big\{\exp(2\gamma_m)-1\big\} \]
\item Lorsque~$(x, y)\in H_\ee$, $x\geq 3$ et $P^+(m)\leq y$, on a
\[ \Psi_m(x, y) = \Lambda_m(x, y) + O\Big(\frac{\Psi(x, y)}{L_\ee(y)g_m(\alpha)}\Big) .\]
\end{enumerate}
\end{lemme}
\begin{remarque}
La Bretèche et Tenenbaum~\cite[formule~(2.22)]{BT05} montrent en fait une estimation pour le rapport~$\Psi(x/d, y)/\Psi(x, y)$ pour un
large domaine en les paramètres~$x, y, d$. On ne fera usage ici que de la majoration énoncée au point~(i).
\end{remarque}

\subsection{Méthode de dispersion}

On reprend dans ce travail la méthode adoptée dans~\cite{FI, BFI}, basée sur un calcul de ``dispersion'' d'après Linnik,
et qui correspond à un calcul de variance empirique. La méthode peut être résumée de la façon suivante :
\begin{enumerate}[(i)]
\item on approche la fonction~$\bfUn_{k\in S(x, y)}$ par des convolutions de la forme~$\sum_{mn\ell=k}\alpha_m\beta_n\lambda_\ell$,
\item par l'inégalité de Cauchy--Schwarz, on se ramène au cas où~$\alpha_m=f(m)$ est une fonction lisse,
\item on détecte la congruence~$mn\ell\equiv a_1\bar{a_2}\mod{r}$ par une formule sommatoire de Poisson, ce qui fait intervenir des sommes de Kloosterman,
\item on majore ces sommes, soit individuellement par la majoration de Weil~\cite{WeilExp}, soit en moyenne grâce à des résultats
de Deshouillers et Iwaniec~\cite{DI}.
\end{enumerate}

Dans tout ce qui suit, on note~$\e(z) = \e^{2\pi i z}$.

La transformée de Fourier d'une fonction~$f:\bfR\to\bfC$ intégrable est définie pour tout~$\eta\in\bfR$ par
\[ \hat f(\eta)=\int_{-\infty}^{\infty}f(\xi)\e(\xi\eta)\dd\xi .\]
Dans ce qui suit, les fonctions dont on considèrera la transformée de Fourier seront toujours~$\cC^\infty$ à support compact. On dispose
dans ce cas de la formule d'inversion
\[ f(\xi)=\int_{-\infty}^{\infty}\hat f(\eta)\e(-\xi\eta)\dd\eta .\]
Enfin, lorsque~$f$ est~$\cC^\infty$ à support dans~$[-M, M]$ avec~$|f^{(j)}|\ll_j M^{-j}$ pour tout~$j\geq 0$,
alors pour tout~$\eta\in\bfR\smallsetminus\{0\}$ et~$j\geq 0$, on a~$|\hat f(\eta)|\ll_j M^{1-j}\eta^{-j}$, ainsi que~$|\hat f(0)| \ll M$. En particulier, on a
\[ \frac1r\sum_{h\in\bfZ\setminus\{0\}}\Big|\hat f\Big(\frac h r\Big)\Big| \ll 1 \qquad (r\in\bfN).\]

Le lemme suivant est une formule sommatoire de Poisson effective, telle que formulée dans~\cite[lemma~2]{BFI}.
\begin{lemme}\label{lemme-poisson}
Soit~$M\geq 1$ et~$f:\bfR_+\to\bfC$ une fonction~$\cC^\infty$ à support compact inclus dans~$[-4M,4M]$,
telle que~$f^{(j)}(x) \ll_j M^{-j}$ lorsque~$x\in\bfR_+$ et~$j\geq 0$. Alors pour tous~$q, a\in\bfN$,~$\ee>0$ et~$H\geq q^{1+\ee}M^{-1}$, on a
\[\sum_{m\equiv a\mod q}f(m)=\frac 1 q \sum_{|h|\leq H}\hat f\Big(\frac h q\Big)\e(-ah/q) + O_\ee(q^{-1}).\]
\end{lemme}

Concernant la majoration de sommes de Kloosterman, on dispose du résultat suivant, qui découle de la majoration de Weil~\cite{WeilExp}.
La version énoncée ici est essentiellement le lemme~4 de~\cite{Fouvry82}.
\begin{lemme}\label{majo-kloo-complete}
Lorsque~$k,b,c,\ell\in\bfN$ et~$D\geq 1$, on a
\begin{equation}\label{somme-exp-weil} \sum_{d\leq D\\(d,c)=1}\e\left(b\frac{\bar{d}}{c}\right)
\ll\log(D+1)\sqrt{(b,c)}\tau(c)c^{1/2}+\frac{(b,c)}{c}D ,\end{equation}
\begin{equation}\label{somme-exp-weil-var} \sum_{d\leq D\\(d,ck)=1\\d\equiv 0\mod \ell}\frac{d}{\vphi(d)}\e\left(b\frac{\bar{d}}{c}\right)
\ll\Big(\log(D+1)^2\sqrt{(b,c)}\tau(c)c^{1/2}+\frac{(b,c)\log(\ell+1)}{\ell c}D\Big)2^{\omega(k)} .\end{equation}
\end{lemme}
Le terme~$\bar{d}$ en argument des exponentielles désigne l'inverse de~$d$ modulo~$c$. La fraction~$b\bar{d}/c$ est définie modulo~$1$, son exponentielle~$e(b\bar{d}/c)$ est donc bien définie.

Le résultat suivant, démontré dans~\cite[theorem~12]{DI}, intervient de façon cruciale. Il s'agit d'une majoration en moyenne
de sommes de Kloosterman. Sa démonstration repose sur la théorie des formes automorphes, et des informations sur le spectre du Laplacien
agissant sur les formes automorphes sur le quotient du demi-plan de Poincaré par les sous-groupes de congruence, $\Gamma\backslash\bfH$
(on réfère le lecteur au chapitre~16 de~\cite{kw} pour plus d'explications à ces sujets). La version que l'on énonce ici correspond au lemma~1
de~\cite{BFI}.
\begin{lemme}\label{kloo-mean}
Soit~$\ee>0$ et~$g_0:\bfR_+\times\bfR_+\to\bfC$ une fonction de classe~$\cC^\infty$ à support compact. Pour tous réels positifs~$C, D, N, R, S$
et toute suite de nombres complexes~$(B_{n,r,s})_{(n, r, s)\in\bfN^3}$, on a
\begin{align*}
&\ \sum_{R<r\leq2R}\sum_{S<s\leq2S}\sum_{1\leq n\leq N}B_{n,r,s}\underset{(rd, sc)=1}{\sum_{C<c\leq2C}\sum_{D<d\leq2D}}
g_0\Big(\frac c C, \frac d D\Big)\e\Big(n\frac{\bar{rd}}{sc}\Big)\\
\ll&\ (CDNRS)^\ee\{CS(RS+N)(C+DR)+C^2DS\sqrt{(RS+N)R}+D^2NRS^{-1}\}^{1/2}\Big\{\sum_{n,r,s}|B_{n,r,s}|^2\Big\}^{1/2}.
\end{align*}
La constante implicite dépend au plus de~$g_0$ et~$\ee$.
\end{lemme}
Selberg a émis la conjecture que le Laplacien agissant sur~$\Gamma\backslash\bfH$ n'a aucune valeur propre dans~$]0, 1/4[$, lorsque~$\Gamma$ est un sous-groupe de congruence ({\it cf.}~\cite{Sarnak} pour plus d'explications à ce sujet).
Si cette conjecture est vraie, on peut ignorer le terme~$C^2DS\sqrt{(RS+N)R}$ dans la majoration du Lemme~\ref{kloo-mean}.
Cela n'a aucune incidence sur le résultat du~Théorème~\ref{thm-BFI-Sxy}.
Par ailleurs, la majoration du Lemme~\ref{kloo-mean} ne dépend pas directement du réel~$\kappa$ tel qu'il soit connu que le spectre du Laplacien
est inclus dans~$\{0\}\cup[\kappa, \infty[$. Deshouillers et Iwaniec utilisent à la place des résultats de densité
sur le spectre du Laplacien~\cite[theorem~6]{DI}.

L'argument que l'on présente ici fait usage d'un résultat récent de Harper~\cite{HarperBV2012}. Plus précisément, il découle des calculs de la section~3.3 de~\cite{HarperBV2012} la majoration suivante, qui concerne les moyennes de sommes de caractères de petits conducteurs sur les entiers friables.
\begin{lemme}\label{lemme-harper}
Il existe des constantes~$\eta, c, \delta>0$ telles que lorsque~$(\log x)^c\leq y\leq x$ et~$Q\leq x$, pour tout~$A>0$ fixé,
\[ \sum_{q\leq Q}\frac1{\vphi(q)}\sum_{\chi \mod{q}\\1<\cond(\chi)\leq x^\eta}\Big|\sum_{n\in S(x, y)}\chi(n)\Big| \ll \Psi(x, y)\big\{H(u)^{-\delta}(\log x)^{-A} + y^{-\delta}\big\} .\]
La constante implicite est effective si~$A<1$.
\end{lemme}

On fera également usage d'une inégalité de grand crible, sous la forme classique suivante ({\it cf.} le theorem~7.13 de~\cite{kw}).
\begin{lemme}\label{gd-crible}
Pour tous entiers~$Q, M, N\geq1$ et toute suite de nombres complexes~$(a_n)_{M<n\leq M+N}$, on a
\[ \sum_{q\leq Q}\frac q{\vphi(q)}\sum_{\chi\mod{q}\\\chi\text{ \emph{primitif}}}\Big|\sum_{M<n\leq M+N}a_n\chi(n)\Big|^2
\leq (N+Q^2-1)\sum_{M<n\leq M+N}|a_n|^2 \]
\end{lemme}

\section{Démonstration du Théorème~\ref{thm-BFI-Sxy}}

Dans toute la suite, pour tous~$r\in\bfN$, $k\mod{r}$ et~$\ee>0$, on note
\begin{equation}\label{def-omega}
\omega_\ee(k ; r):=\sum_{\chi\prim\\\cond(\chi)\leq x^\ee\\\cond(\chi)|r}\chi(k)
\end{equation}
où la somme est sur l'ensemble des caractères primitifs de Dirichlet de conducteur inférieur à~$x^\ee$ et divisant~$r$.
En particulier, cela inclut toujours le caractère trivial $\chi=\bfUn$, correspondant au module~$1$.
\begin{theoreme}\label{thm-BFI-MNL}
Pour tout~$\ee>0$ suffisament petit, lorsque~$a_1, a_2\in\bfZ\setminus\{0\}, (a_1, a_2)=1$,
et lorsque~$(\alpha_m)$, $(\beta_n)$, $(\lambda_\ell)$ sont trois suites de nombres complexes
de modules inférieurs ou égaux à~$1$, de supports respectifs dans les entiers de~$]M, 2M]$, $]N, 2N]$, $]L, 2L]$, avec~$x:=MNL$,
il existe~$\delta>0$ pouvant dépendre de~$\ee$ tel que lorsque l'un quelconque des deux ensembles de conditions sur~$M, N, L, R$ suivants est vérifié :
\begin{equation}\label{cond-thm-pcp-kloomean}
\left\{\begin{array}{l}
|a_1|, |a_2|\leq x^\delta, \qquad x^\ee\leq N, \qquad NL\leq x^{2/3-5\ee}, \qquad L\leq x^{-\ee}M, \\
M\leq R\leq x^{-\ee}NL, \qquad N^2L^3 \leq x^{1-\ee} R, \qquad N^5L^2 \leq x^{2-\ee}, \qquad N^4L^3\leq x^{2-\ee},
\end{array}\right.
\end{equation}
ou
\begin{equation}\label{cond-thm-pcp-klooweil}
\left\{\begin{array}{l}
|a_1|\leq x^{1-\ee}, |a_2|\leq x^\delta \qquad \qquad NL\leq x^{2/3-5\ee}, \\
M\leq R\leq x^{-\ee}NL, \qquad N^3L^4\leq x^{2-\ee}, \qquad N^6L^5\leq x^{4-\ee}R^{-2},
\end{array}\right.
\end{equation}
on ait la majoration
\begin{equation}\label{estim-MNL}
\sum_{R<r\leq2R \\ (r, a_1a_2)=1} \Bigg|
\mathop{\sum_{m}\sum_{n}\sum_{\ell}}_{mn\ell\equiv a_1\bar{a_2}\mod{r}} \alpha_m\beta_n\lambda_\ell -
\frac{1}{\vphi(r)}\mathop{\sum_{m}\sum_{n}\sum_{\ell}}_{(mn\ell, r)=1} \alpha_m\beta_n\lambda_\ell \omega_\ee(mn\ell\bar{a_1}a_2 ; r)
\Bigg| \ll_\ee x^{1-\delta} .\end{equation}
\end{theoreme}
En particulier, on a sous les mêmes hypothèses la majoration
\[ \sum_{R<r\leq2R \\r~\text{premier}\\r\nmid a_1a_2} \Bigg|
\mathop{\sum_{m}\sum_{n}\sum_{\ell}}_{mn\ell\equiv a_1\bar{a_2}\mod{r}} \alpha_m\beta_n\lambda_\ell -
\frac{1}{\vphi(r)}\mathop{\sum_{m}\sum_{n}\sum_{\ell}}_{r\nmid mn\ell} \alpha_m\beta_n\lambda_\ell
\Bigg| \ll_\ee x^{1-\delta} .\]
Le Théorème~\ref{thm-BFI-MNL} est à rapprocher du theorem~4 de~\cite{BFI} et du théorème~2 de~\cite{Fouvry82}, qui énoncent que la majoration
\begin{equation}\label{estim-MNL-before} \sum_{R<r\leq2R \\ (r, a_1)=1} \Bigg|
\mathop{\sum_{m}\sum_{n}\sum_{\ell}}_{mn\ell\equiv a_1\mod{r}} \alpha_m\beta_n\lambda_\ell -
\frac{1}{\vphi(r)}\mathop{\sum_{m}\sum_{n}\sum_{\ell}}_{(mn\ell, r)=1} \alpha_m\beta_n\lambda_\ell
\Bigg| \ll_{\ee,A} x/(\log x)^A \end{equation}
est valable sous les conditions du Théorème~\ref{thm-BFI-MNL} et sous une hypothèse supplémentaire de type~{\it Siegel-Walfisz}, que~$(\beta_n)$ est bien répartie dans les progressions arithmétiques de petits modules\footnote{Dans~\cite{BFI}, les auteurs font d'autres hypothèses simplificatrices, nommément (A_4) et~(A_5). Les auteurs mentionnent (p.220) que ces hypothèses ne sont pas cruciales, ce que les calculs du présent travail mettent en évidence.}. Le gain dans la majoration~\eqref{estim-MNL} par rapport à~\eqref{estim-MNL-before} provient du fait que l'on considère uniquement
la contribution des caractères de conducteurs~$>x^\ee$. Cela est la principale nouveauté de ce travail par rapport à~\cite{BFI} et~\cite{Fouvry82}.
Une telle réduction est admissible pour l'application aux entiers friables en vertu du Lemme~\ref{lemme-harper} de Harper précité.

Les conditions limitantes pour la taille de~$R$ dans le système~\eqref{cond-thm-pcp-kloomean} sont heuristiquement
\[ R\leq x^{-\ee}NL, \qquad N^2L^3 \leq x^{1-\ee} R, \qquad N^4L^3\leq x^{2-\ee} .\]
Elles impliquent en effet~$R\leq x^{3/5-8\ee/5}$, et l'égalité est atteinte pour les valeurs~$N = x^{1/5+4\ee/5}$, $L = x^{2/5-7\ee/5}$, $M = x^{2/5+3\ee/5}$.
De même, les conditions limitantes pour la taille de~$R$ dans le système~\eqref{cond-thm-pcp-klooweil} sont
\[ R\leq x^{-\ee}NL, \qquad N^3L^4\leq x^{2-\ee}, \qquad N^6L^5\leq x^{4-\ee}R^{-2}, \]
elles impliquent~$R\leq x^{6/11-\ee}$, et l'égalité est atteinte pour~$N=x^{2/11+\ee}$, $L=x^{4/11-\ee}$, $M=x^{4/11}$.
Cela est à l'origine des exposants~$3/5$ et~$6/11$ dans le Théorème~\ref{thm-BFI-Sxy}.

On remarque que, contrairement à~\cite[théorème~2]{Fouvry82} et~\cite[theorem~4]{BFI}, on ne fait pas d'hypothèse
sur la bonne répartition modulo~$r$ des suites étudiées. Une hypothèse comme l'hypothèse~(A_2) de~\cite{BFI} ne serait pas pertinente
dans notre étude puisque l'on considère uniquement la contribution des caractères de grands conducteurs.

La démonstration proposée suit dans une large mesure celles de~\cite[théorème~2]{Fouvry82} et~\cite[theorem~4]{BFI},
On s'inspire des calculs réalisés dans~\cite[section~3]{FI} afin de ne pas recourir à certaines hypothèses contraignantes,
comme l'hypothèse~(A_4) de~\cite{BFI}.

On note~$(u_k)$ la suite définie pour tout~$k\in\bfN$ par~$u_k:=\sum_{n\ell=k}\beta_n\lambda_\ell$.
Cette suite est à support dans les entiers de~$]K,4K]$ avec~$K:=NL$, et vérifie~$|u_k|\leq\tau(k)$ ainsi que~$\sum_{k}|u_k|\ll K$.
On note également
\[ \cX = \cX_\ee := \big\{\chi\prim\ |\  \cond(\chi)\leq x^\ee\} .\]
On a~$|\cX| \sim x^{2\ee}/(2\zeta(2)^2)$ lorsque~$x\to\infty$, {\it cf.}~\cite[formule~(3.7)]{kw}.

\subsection{Calcul de dispersion}

On désigne par~$\Delta(M,N,L,R)$ le membre de gauche de~\eqref{estim-MNL}. On a par l'inégalité de Cauchy--Schwarz
\begin{equation}\label{majo-delta-init}
\Delta(M,N,L,R)^2 \leq MR\sum_{R<r\leq 2R\\(r, a_1a_2)=1} \sum_{M<m\leq 2M\\(m, r)=1} \Bigg|\sum_{k\equiv a_1\bar{a_2m}\mod{r}} u_k
-\frac{1}{\vphi(r)}\sum_{(k, r)=1} u_k\omega_\ee(mk\bar{a_1}a_2 ; r) \Bigg|^2 .\end{equation}
Étant donnée une fonction~$\cC^\infty$ $\Phi_0:\bfR_+\to\bfR_+$, à support compact inclus dans~$[1/2, 3]$
et majorant la fonction indicatrice de l'intervalle~$[1, 2]$, on pose~$f(m) := \Phi_0(m/M)$.
Alors le membre de droite de l'expression précédente est majoré par
\begin{equation}\label{dispersion}\begin{aligned}
&\ MR \sum_{R<r\leq 2R\\(r, a_1a_2)=1} \sum_{(m, r)=1} f(m)
\Bigg|\sum_{k\equiv a_1\bar{a_2m}\mod{r}}u_k-\frac{1}{\vphi(r)}\sum_{(k, r)=1} u_k\omega_\ee(mk\bar{a_1}a_2 ; r) \Bigg|^2 \\
=&\ MR(\cS_1-2\Re \cS_2+\cS_3)
\end{aligned}\end{equation}
avec
\[ \cS_1 := \sum_{R<r\leq 2R\\(r,a_1a_2)=1}\sum_{(m,r)=1}f(m)\Bigg|\sum_{k\equiv a_1\bar{a_2m}\mod{r}}u_k\Bigg|^2, \]
\[ \cS_2 := \sum_{R<r\leq 2R\\(r,a_1a_2)=1}\frac{1}{\vphi(r)}\sum_{(m,r)=1}f(m)\summ_{(k_1, r)=1\\k_2\equiv a_1\bar{a_2m}\mod{r}}u_{k_1}
\omega_\ee(mk_1\bar{a_1}a_2;r)\bar{u_{k_2}},\]
\[ \cS_3 := \sum_{R<r\leq 2R\\(r,a_1a_2)=1}\frac{1}{\vphi(r)^2}\sum_{(m,r)=1}f(m)\Bigg|\sum_{(k,r)=1}u_{k}\omega_\ee(mk\bar{a_1}a_2;r)\Bigg|^2 .\]
On évalue successivement~$\cS_3, \cS_2$ puis~$\cS_1$.

\subsection{Estimation de~$\cS\sb3$}
On a
\[ \cS_3 = \sum_{R<r\leq2R\\(r,a_1a_2)=1}\frac1{\vphi(r)^2}\summ_{\chi_1, \chi_2\in\cX\\\cond(\chi_1)|r\\\cond(\chi_2)|r}\bar{\chi_1}\chi_2(a_1\bar{a_2})
\summ_{k_1, k_2\\(k_1k_2,r)=1}\chi_1(k_1)u_{k_1}\bar{\chi_2(k_2)u_{k_2}}\sum_{(m, r)=1}f(m)\chi_1\bar{\chi_2}(m) .\]
Le Lemme~\ref{lemme-poisson} permet d'évaluer la somme en~$m$. Notant~$H := x^{\ee}RM^{-1}$, pour tous~$\chi_1,\chi_2\in\cX$
avec~$\cond(\chi_1)|r$ et~$\cond(\chi_2)|r$, on a
\begin{align*}
\sum_{(m, r)=1}f(m)\chi_1\bar{\chi_2}(m)=&\ \hat f(0)\sum_{0<b<r\\(b,r)=1}\chi_1\bar{\chi_2}(b)
+\frac 1 r\sum_{0<|h|<H}\hat f\Big(\frac h r\Big)\sum_{0<b<r\\(b,r)=1}\chi_1\bar{\chi_2}(b)\e\Big(\frac{-bh}r\Big)+O(1).
\end{align*}
La somme sur~$b$ est une somme de Gauss, pour laquelle on dispose des lemmes~3.1 et~3.2 de~\cite{kw}.
Par ailleurs le module de~$\chi_1\bar{\chi_2}$ est un diviseur de~$r$ qui est inférieur à~$x^{2\ee}$. On obtient donc
\[ \Big|\sum_{0<b<r\\(b,r)=1}\chi_1\bar{\chi_2}(b)\e\Big(\frac{-bh}{r}\Big)\Big|\leq \cond(\chi_1\bar{\chi_2})^{1/2}\sum_{d|(h,r)}d .\]
Ainsi, en utilisant la majoration~$\|\hat f\ \|_\infty\ll M$, on a
\[ \frac 1 r\sum_{0<|h|<H}\hat f\Big(\frac h r\Big)\sum_{0<b<r\\(b,r)=1}\chi_1\bar{\chi_2}(b)\e\Big(\frac{-bh} r\Big)
\leq \frac{x^\ee}r\sum_{d|r}d\sum_{|h|\leq H/d}\Big|\hat f\Big(\frac{dh}{r}\Big)\Big| \ll x^{2\ee}\tau(r) .\]
En reportant dans~$\cS_3$, on obtient
\begin{equation}\label{estim-S3} \cS_3 = \hat f(0)X_3+R_3 \end{equation}
avec
\[ X_3 := \sum_{R<r\leq 2R\\(r,a_1a_2)=1}\frac 1 r\sum_{0<b<r\\(b,r)=1}\Bigg|\frac 1{\vphi(r)}\sum_{(k,r)=1}u_k\omega_\ee(k\bar b ; r)\Bigg|^2, \]
\[ R_3\ll x^{2\ee}|\cX|^2\sum_{R<r\leq 2R\\(r,a_1a_2)=1}\frac{\tau(r)}{\vphi(r)^2}\Bigg(\sum_{(k,r)=1}|u_k|\Bigg)^2\ll x^{6\ee}K^2(\log R)R^{-1} .\]
Chacun des deux systèmes de conditions~\eqref{cond-thm-pcp-kloomean} et~\eqref{cond-thm-pcp-klooweil}
implique~$K\leq x^{2/3}$, quitte à supposer~$\ee$ suffisamment petit on obtient~$R_3=O(x^{1-\ee}KR^{-1})$.

\subsection{Estimation de~$\cS\sb2$}
On a
\[ \cS_2 = \sum_{\chi\in\cX}\sum_{R<r\leq2R\\(r,a_1a_2)=1\\\cond(\chi)|r}\frac1{\vphi(r)}\sum_{(m,r)=1}f(m)\summ_{(k_1,r)=1\\k_2\equiv a_1\bar{a_2m}\mod{r}}
\chi(k_1m\bar{a_1}a_2)u_{k_1}\bar{u_{k_2}} .\]
Pour des indices~$\chi, r, m, k_1, k_2$ de cette somme, on a~$\chi(k_1m\bar{a_1}a_2) = \chi(k_1)\bar{\chi(k_2)}$, ainsi
\[ \cS_2 = \sum_{\chi\in\cX}\sum_{k_1}\chi(k_1)u_{k_1}\sum_{k_2}\bar{\chi(k_2)u_{k_2}}
\sum_{R<r\leq 2R\\(r,a_ik_i)=1\\\cond(\chi)|r}\frac 1{\vphi(r)}\sum_{m\equiv a_1\bar{a_2k_2}\mod{r}}f(m).\]
D'après le Lemme~\ref{lemme-poisson}, en posant~$H=x^{\ee}RM^{-1}$, la somme sur~$m$ vaut
\[ \sum_{m\equiv a_1\bar{a_2k_2}\mod{r}}f(m) = \frac 1 r\sum_{|h|\leq H}\hat f\Big(\frac h r\Big)\e(-a_1h\bar{a_2k_2}/r)+O_\ee(1/r).\]
En isolant la contribution du terme d'indice~$h=0$ et en reportant cela dans~$\cS_2$, on obtient :
\[ \cS_2 = \hat f(0)X_2 + R_2 +O(|\cX|K^2R^{-1})\]
avec
\[ X_2:=\sum_{R<r\leq 2R\\(r,a_i)=1}\frac 1{r\vphi(r)}\summ_{(k_1,r)=1\\(k_2,r)=1}u_{k_1}\bar{u_{k_2}}\omega_\ee(k_1\bar{k_2} ; r) ,\]
\[ R_2:=\sum_{\chi\in\cX}\sum_{k_1}\chi(k_1)u_{k_1}\sum_{k_2}\bar{\chi(k_2)u_{k_2}}
\sum_{R<r\leq 2R\\(r,a_ik_i)=1\\\cond(\chi)|r}\frac 1{r\vphi(r)}\sum_{1\leq|h|\leq H}\hat f\Big(\frac h r\Big)\e\Big(\frac{-a_1h\bar{a_2k_2}} r\Big) .\]
Les entiers~$a_2k_2$ et~$r$ étant premiers entre eux, on a l'égalité
\[ \frac{\bar{a_2k_2}}{r}\equiv-\frac{\bar{r}}{a_2k_2}+\frac 1{a_2k_2r} \quad \mod{1}.\]
On a donc
\begin{equation}\label{majo-R2} R_2\leq R^{-2}\sum_{\chi\in\cX}\summ_{k_1, k_2}|u_{k_1}u_{k_2}|\sum_{1\leq|h|\leq H}\Bigg|\sum_{R<r\leq2R\\(r, a_ik_i)=1\\\cond(\chi)|r}
\frac r{\vphi(r)}\hat f\Big(\frac h r\Big)\e\Big(\frac{a_1h\bar r}{a_2k_2}-\frac{a_1h}{a_2rk_2}\Big)\Bigg|
.\end{equation}
On effectue une intégration par parties afin d'éliminer le terme~$\hat f(h/r)\e(-a_1h/(a_2rk_2))$. Le membre de droite de~\eqref{majo-R2} s'écrit
\[ R^{-2}\sum_{z\in\cZ}\lambda(z)\Big|\sum_{R<r\leq2R}g_1(z, r)g_2(z, r)\Big| \]
où on a posé
\[ \cZ := \big\{(\chi, k_1, k_2, h)\ |\ \chi\in\cX, K<k_i\leq 4K, 1\leq|h|\leq H \big\} \]
et pour tout~$z=(\chi,k_1,k_2,h)\in\cZ$,
\[ \lambda(z):=|u_{k_1}u_{k_2}|, \qquad g_1(z, r) := \hat f\Big(\frac h r\Big)\e\Big(-\frac{a_1h}{a_2rk_2}\Big),
\qquad g_2(z, r) := \bfUn_{(r, a_ik_i)=1}\bfUn_{\cond(\chi)|r}\frac r{\vphi(r)}\e\Big(\frac{a_1h\bar r}{a_2k_2}\Big) .\]
Posant~$G_2(z, \xi) := \sum_{R<r'\leq\xi}g_2(z, r')$, une intégration par parties fournit
\begin{align*}
\sum_{z\in\cZ}\lambda(z)\Big|\sum_{R<r\leq2R}g_1(z, r)g_2(z, r)\Big| &= \sum_{z\in\cZ}\lambda(z)\Big|\int_{R+}^{2R+}g_1(z,\xi)\dd G_2(z,\xi)\Big| \\
&\leq \sup_{z\in\cZ\\R<\xi\leq2R}|g_1(z,\xi)|\times\sum_{z\in\cZ}\lambda(z)\big|G_2(z, 2R)\big| \\
&\qquad+ R\sup_{z\in\cZ\\R<\xi\leq2R}\Big|\frac{\partial g_1}{\partial \xi}(z,\xi)|\times\sup_{R<\xi\leq2R}\sum_{z\in\cZ}\lambda(z)\big|G_2(z, \xi)\big|
\end{align*}
Pour tout~$\xi\in]R,2R]$ et~$z=(\chi, k_1, k_2, h)\in\cZ$, la majoration~\eqref{somme-exp-weil-var} fournit
\[ G_2(z, \xi) \ll x^{\ee/2}\big\{(a_1h,a_2k_2)^{1/2}k_2^{1/2}+(a_1h,a_2k_2)Rk_2^{-1}\big\} ,\]
on obtient donc, en utilisant l'inégalité~$\sum_{t\leq T}(at, b)\leq (a, b)\tau(b)T$ valable pour tous~$b\in\bfN, T\geq 1$,
\[ \sum_{z\in\cZ}\lambda(z)\big|G_2(z, \xi)\big| \ll x^\ee |\cX|HK^{5/2} + x^\ee|\cX|HKR .\]
Par ailleurs, ~$|g_1(z, \xi)|\ll M$ et~$\partial g_1/\partial \xi(z, \xi) \ll HM^2R^{-2}$ dès lors que l'on a~$|a_1|\leq x$.
Ainsi, sous les hypothèses~\eqref{cond-thm-pcp-kloomean} ou~\eqref{cond-thm-pcp-klooweil},
on a~$R_2 \ll x^{6\ee}\big\{K^{5/2}R^{-1}+K\big\}\ll x^{1-\ee}KR^{-1}$, ce qui fournit
\begin{equation}\label{estim-S2} \cS_2 = \hat f(0)X_2+O(x^{1-\ee}KR^{-1}) .\end{equation}

\subsection{Estimation de~$\cS\sb1$}

En séparant les sommants de~$\cS_1$ selon la valeur de~$(k_1,k_2)$, on a
\[ \cS_1=\sum_{v\geq 1}\sum_{R<r\leq 2R\\(r,a_iv)=1}\sum_{(m,r)=1}f(m)\summ_{k_1\equiv a_1\bar{a_2m}\mod{r}\\k_2\equiv a_1\bar{a_2m}\mod{r}\\(k_1,k_2)=v}u_{k_1}\bar{u_{k_2}}
= \sum_{v\geq 1}S(v) \]
avec pour tout~$v\geq 1$,
\[ S(v):=\sum_{R<r\leq 2R\\(r,a_iv)=1}\sum_{(m,r)=1}f(m)\summ_{k_1\equiv a_1\bar{a_2m}\mod{r}\\k_2\equiv a_1\bar{a_2m}\mod{r}\\(k_1,k_2)=v}u_{k_1}\bar{u_{k_2}} .\]
Soit~$\eta>0$. La contribution des indices~$v\geq x^{\eta}$ est majorée par
\begin{align*}
\sum_{v\geq x^\eta}S(v)\leq&\ \sum_{R<r\leq 2R\\(r,a_1a_2)=1}\sum_{(m,r)=1}f(m)\sum_{v\geq x^\eta\\(v,r)=1}\Bigg(\sum_{k\equiv a_1\bar{a_2m}\mod{r}\\v|k}|u_k|\Bigg)^2.\end{align*}
La somme sur~$k$ est majorée par~$x^{\eta/4}\big\{K/(vR)+1\big\}\ll x^{-3\eta/4}K/R$ dès que~$3\eta/4<\ee$. On a donc
\begin{align*}
\sum_{v\geq x^\eta}S(v)\ll&\ \frac{K}{Rx^{3\eta/4}}\sum_{R<r\leq 2R\\(r,a_1a_2)=1}\sum_{(m,r)=1}f(m)\sum_{v\geq x^\eta\\(v,r)=1}\sum_{k\equiv a_1\bar{a_2m}\mod{r}\\v|k}|u_k|\\
\ll&\ \frac{K}{Rx^{3\eta/4}}\sum_{k}|u_k|\tau(k)\sum_{m\\a_2mk\neq a_1}f(m)\tau(|a_2mk-a_1|) + Kx^{-3\eta/4}\tau(|a_1|)^3\\
\ll&\ x^{1-\eta/2}KR^{-1}
\end{align*}
quitte à supposer~$\eta\leq\ee$. On fixe à présent un entier~$v\leq x^{\eta}$. On écrit de façon unique $k_1=vd_1e_1k_1''$ avec~$d_1|v^\infty$ (c'est-à-dire~$p|d_1\Rightarrow p|v$), $e_1|a_2^\infty$ et~$(k_1'', va_2)=1$, ainsi que~$k_2=ve_2k_2'$ avec~$e_2=(k_2/v, a_2)$. La contribution à~$S(v)$ des indices~$k_1, k_2$ pour lesquels~$d_1>x^{\eta}$ est
\begin{align*}
\leq&\ \sum_{d_1>x^{\eta}\\d_1|v^\infty}\sum_{R<r\leq2R\\(r, a_1a_2v)=1}\sum_{(m, r)=1}f(m)
\summ_{vd_1|k_1, v|k_2\\k_1\equiv k_2\equiv a_1\bar{a_2m}\mod{r}}|u_{k_1}u_{k_2}| \\
\ll&\ x^{\eta/4}\sum_{d_1>x^{\eta}\\d_1|v^\infty}\sum_{R<r\leq2R\\(r,a_1a_2v)=1}\sum_{(m, r)=1}f(m)\sum_{K/(vd_1)<\tilde{k_1}\leq2K/(vd_1)\\\tilde{k_1}\equiv a_1\bar{a_2vd_1m}\mod{r}}
\sum_{K/v<\tilde{k_2}\leq2K/v\\\tilde{k_2}\equiv d_1\tilde{k_1}\mod{r}}1\\
\ll&\ KR^{-1}x^{\eta/4}\sum_{d_1>x^{\eta}\\d_1|v^\infty}\sum_{m}f(m)\sum_{K/(vd_1)<\tilde{k_1}\leq2K/(vd_1)\\a_2vd_1\tilde{k_1}m\neq a_1}\tau(|a_2vd_1\tilde{k_1}m-a_1|)
+ K\tau(|a_1|)^3 \\
\ll&\ x^{1+\eta/2}KR^{-1}\sum_{d_1>x^{\eta}\\d_1|v^\infty}d_1^{-1} + Kx^{\eta}\\
\ll&\ x^{1-\eta/4}KR^{-1}
\end{align*}
la dernière majoration étant uniforme pour~$v\leq x$. On montre de même que la contribution à~$S(v)$ des indices vérifiant~$e_1>x^\eta$ ou~$e_2>x^\eta$
est~$O(x^{1-\eta/4}KR^{-1})$. On se retreint donc dorénavant à~$\max\{d_1, e_1, e_2\}\leq x^\eta$. On note pour tous~$d_1|v^\infty$, $e_1|a_2^\infty$ et~$e_2|a_2$ :
\[ S(v ; d_1, e_1, e_2):=\sum_{R<r\leq2R\\(r,a_1a_2v)=1}\sum_{(m, r)=1}f(m)
\summ_{k_1\equiv k_2\equiv a_1\bar{a_2m}\mod{r}\\(k_1,k_2)=v, ve_1d_1|k_1, ve_2|k_2\\(k_1/(vd_1e_1), va_2)=(k_2/(ve_2), a_2)=1}u_{k_1}\bar{u_{k_2}} .\]
En appliquant le Lemme~\ref{lemme-poisson} à la somme sur~$m$, on obtient
\[ S(v ; d_1, e_1, e_2)=\hat f(0)X_1(v; d_1, e_1, e_2)+R(v ; d_1, e_1, e_2)+T(v ; d_1, e_1, e_2) \]
avec les notations~$H:=x^\eta RM^{-1}$ et
\[ \cK = \cK(r, v ; d_1, e_1, e_2) := \left\{(k_1, k_2)\in\bfN\ \Bigg|\ \begin{array}{l}k_1\equiv k_2\mod{r}, (k_1, k_2)=v, ve_1d_1|k_1, ve_2|k_2, \\
(k_1, r) = (k_1/(vd_1e_1), va_2)=(k_2/(ve_2), a_2)=1 \end{array} \right\},\]
\begin{align} \notag X_1(v ; d_1, e_1, e_2):=&\ \sum_{R<r\leq 2R\\(r,a_1a_2v)=1}\frac 1 r\summ_{(k_1, k_2)\in\cK} u_{k_1}\bar{u_{k_2}}, \\
\label{def-Rv} R(v ; d_1, e_1, e_2):=&\ \sum_{R<r\leq 2R\\(r,a_1a_2v)=1}\frac 1 r\summ_{(k_1, k_2)\in\cK}u_{k_1}\bar{u_{k_2}}
\sum_{1\leq|h|\leq H}\hat f\Big(\frac h r\Big)\e\Big(\frac{-ha_1\bar{a_2k_1}}{r}\Big) \quad ,\\
\notag T(v ; d_1, e_1, e_2)\ll&\ \sum_{R<r\leq 2R\\(r,v)=1}\frac 1 r\summ_{(k_1, k_2)\in\cK}|u_{k_1}u_{k_2}| .\end{align}
On a de plus, quitte à supposer~$\eta<\ee$ (de sorte que~$K>x^\eta$),
\begin{align*} \sum_{v\geq 1}\sum_{d_1|v^\infty}\sum_{e_1|a_2^\infty}\sum_{e_2|a_2}T(v ; d_1, e_1, e_2)\ll&\ x^\eta K^2R^{-1}+K, \\
\sum_{v\leq x^\eta}\summm_{\max\{d_1, e_1, e_2\}>x^\eta\\d_1|v^\infty, e_1|a_2^\infty, e_2|a_2}X_1(v ; d_1, e_1, e_2)
\ll&\ K^2R^{-1}x^{-\eta/2} , \\
\sum_{v>x^\eta}\sum_{d_1|v^\infty}\sum_{e_1|a_2^\infty}\sum_{e_2|a_2}X_1(v ; d_1, e_1, e_2) \ll&\ K^2R^{-1}x^{-\eta/2} + Kx^{\eta/2} . \end{align*}
Au final, on obtient
\begin{equation}\label{eg-S1}
\cS_1 = \hat f(0)X_1 + \sum_{v\leq x^{\eta}}\summm_{d_1, e_1, e_2\leq x^\eta\\d_1|v^\infty, e_1|a_2^\infty, e_2|a_2}R(v ; d_1, e_1, e_2) + O(x^{1-\eta/4}KR^{-1})
\end{equation}
avec
\[ X_1 := \sum_{R<r\leq 2R\\(r,a_1a_2)=1}\frac 1 r\sum_{0<b<r\\(b,r)=1}\Bigg|\sum_{k\equiv b\mod{r}}u_{k}\Bigg|^2 .\]
Il reste à étudier le second terme du membre de droite de~\eqref{eg-S1}. La suite~$(u_{vk'})_{k'}$ est une combinaison linéaire de~$\tau(v)$ produits de convolutions de suites de la forme~$(\beta_{d n'})_{n'}$ et~$(\lambda_{d \ell'})_{\ell'}$ pour différents entiers~$d$ divisant~$v$. Afin de clarifier les notations, on  la proposition suivante.

\begin{prop}\label{prop-bfi-final}
Soit~$\ee>0$ fixé suffisamment petit. Lorsque~$a\in\bfZ\setminus\{0\}$ et~$v, d_1, d_2\in\bfN$,
pour tous réels~$M, K, N, L, H, R, \cE$ supérieurs à~$1$, toutes suites~$(u_{k})$, $(\beta_n)$, $(\lambda_\ell)$
de supports respectifs dans les entiers de~$]K, 4K]$, $]N, 2N]$, $]L, 2L]$, vérifiant pour tous~$k, n, \ell$,
\[ |u_k|\leq \tau(k), \quad \max\{|\beta_n|, |\lambda_\ell|\}\leq 1 ,\]
\[ (k,vd_1d_2)>1\Rightarrow u_k=0, \qquad (n\ell,vd_1)>1\Rightarrow \beta_n\lambda_\ell=0, \]
et toute fonction lisse~$\Phi_0:\bfR_+\to\bfR_+$ à support inclus dans~$[1/2,3]$, posant~$f(m) := \Phi_0(m/M)$,
la majoration suivante :
\begin{equation}\label{prop-bfi-final-res}
\sum_{R<r\leq 2R\\(r,avd_1d_2)=1}\frac 1 r\summm_{d_1k\equiv d_2n\ell\mod{r}\\(d_1k,d_2n\ell)=1}u_{k}\beta_n\lambda_\ell
\sum_{1\leq|h|\leq H}\hat f\Big(\frac h r\Big)\e\Big(\frac{-ha\bar{vd_1d_2k}}{r}\Big) \ll (MKNL)^{10\ee} MKNLR^{-1}\cE^{-1}
\end{equation}
est valable sous l'un quelconque des deux ensembles de conditions suivants :
\begin{equation}\label{cond-prop-kloomean}\begin{aligned}
\left\{
\begin{array}{l}
|avd_1d_2| \leq (KNL)^\ee, \quad NL\leq d_2K, \quad H\leq R^{1+\ee}M^{-1}, \\
K\leq LN^2, \quad R^{\ee}K\leq NM, \quad M\leq R\leq K, \quad NL\leq KM, \quad R\leq M^2NL\cE^{-1}, \\
K\leq M^{1/2}N^{1/2}R^{1/2}\cE^{-1}, \quad K^{1/4}N^{1/2}\leq M^{1/2}L^{1/4}\cE^{-1}, \quad K^{1/2}\leq M^{1/2}L^{1/4}\cE^{-1},
\end{array}
\right.
\end{aligned}\end{equation}
ou
\begin{equation}\label{cond-prop-klooweil}\begin{aligned}
\left\{
\begin{array}{l}
|vd_1d_2| \leq (KNL)^\ee, \quad |a|R^\ee\leq vd_1d_2^2M\big\{K+NL\big\}, \quad H\leq R^{1+\ee}M^{-1}, \\
M\leq R\leq K, \quad K\leq MN^{1/2}\cE^{-1}, \quad N^{1/2}L^{1/4}\leq MR^{-1/2}\cE^{-1}
\end{array}
\right.
\end{aligned}\end{equation}
\end{prop}

On insiste sur le fait que dans les hypothèses de cette proposition, $K$ n'est pas nécessairement égal à~$NL$. Dans la somme du membre de gauche de~\eqref{prop-bfi-final-res}, les deux conditions~$d_1k\equiv d_2n\ell\mod{r}$ et~$(k, d_2n\ell)=1$ impliquent~$(k, r)=1$.

Admettons temporairement la Proposition~\ref{prop-bfi-final}. Pour tout indice~$k_2$ dans la définition~\eqref{def-Rv}, on note~$d_2 = (k_2/(ve_2), v)$ et~$k_2 = vd_2e_2k_2''$. On a
\[ u_{k_2}=\sum_{\delta_1\delta_2=ve_2}\sum_{n'\ell'=d_2k_2''\\(\ell', \delta_1)=1}\beta_{\delta_1n'}\gamma_{\delta_2\ell'}
=\sum_{\delta_1\delta_2=ve_2}\sum_{\delta_3\delta_4=d_2\\(\delta_4,\delta_1)=1}\sum_{n''\ell''=k_2''\\(\ell'',\delta_1\delta_3)=1}\beta_{\delta_1\delta_3n''}\gamma_{\delta_2\delta_4\ell''} \]
On obtient donc
\begin{equation}\label{expr-R-Rtilde} R(v ; d_1, e_1, e_2)
= \sum_{d_2|v\\(d_2, a_2)=1}\sum_{\delta_1\delta_2=ve_2}\sum_{\delta_3\delta_4=d_2\\(\delta_4,\delta_1)=1}
\tilde R(v; d_1, d_2, \delta_1, \delta_3 ; e_1, e_2) \end{equation}
avec
\[ \tilde R(v; d_1, d_2, \delta_1, \delta_3, e_1, e_2) := \sum_{R<r\leq 2R\\(r, a_1a_2v)=1}\summm_{d_1e_1k''\equiv d_2e_2n''\ell''\mod{r}\\(d_1e_1k'',d_2e_2n''\ell'')=1}
\tilde u_{k''}\tilde\beta_{n''}\tilde\lambda_{\ell''}\sum_{1\leq|h|\leq H}\hat f\Big(\frac h r\Big)\e\Big(\frac{-a_1h\bar{a_2vd_1e_1k''}}{r}\Big), \]
\[ \tilde u_{k''} := \bfUn_{(k'', va_2)=1}u_{vd_1e_1k''}, \quad \tilde\beta_{n''}:=\bfUn_{(n'',a_2v/d_2)=1}\bar{\beta_{\delta_1\delta_3n''}},
\quad \tilde\lambda_{\ell''}:=\bfUn_{(\ell'', \delta_1\delta_3a_2v/d_2)=1}\bar{\lambda_{\delta_2\delta_4\ell''}} .\]
Pour chaque choix d'indices dans la somme~\eqref{expr-R-Rtilde}, on applique la Proposition~\ref{prop-bfi-final} avec les paramètres
\[ v \gets \frac{a_2v}{e_2d_2}, \quad d_1 \gets e_1d_1, \quad d_2\gets e_2d_2,
\quad K\gets K/(vd_1e_1), \quad N\gets N/(\delta_1\delta_3), \quad L\gets L/(\delta_2\delta_4). \]
On obtient pour un certain~$\eta'>0$ la majoration
\[ \tilde R(v; d_1, d_2, \delta_1, \delta_3 ; e_1, e_2) \ll x^{1-\eta'}KR^{-1} \]
lorsque l'un des deux ensembles de conditions suivants est satisfait :
\[
\left\{
\begin{array}{l}
|a_1a_2vd_1e_1|\leq x^{\eta'}, x^\eta \leq R^{\eta'}, \quad x^{5\eta'}L\leq M, \\
M\leq R\leq Kx^{-5\eta'}, \quad x^{5\eta'}\leq \min\{M, N\}, \quad x^{20\eta'}R\leq M^2K, \\
x^{20\eta'}N^{1/2}L\leq M^{1/2}R^{1/2}, \quad x^{20\eta'}N^{3/4}\leq M^{1/2}, \quad x^{20\eta'}N^{1/2}L^{1/4}\leq M^{1/2},
\end{array}
\right.
\]
ou
\[
\left\{
\begin{array}{l}
|a_2vd_1e_1|\leq x^{\eta'}, \quad |a_1|\leq x^{1-5\eta'}, \quad  x^\eta \leq R^{\eta'}, \\
M\leq R\leq Kx^{-5\eta'}, \quad x^{30\eta'}N^{1/2}L\leq M, \quad x^{30\eta'}N^{1/2}L^{1/4}\leq MR^{-1}.
\end{array}
\right.
\]
Cela implique~$\cS_1 - \hat f(0)X_1 \ll  (x^{1+4\eta-\eta'}+x^{1-\eta/4})KR^{-1}$.
Lorsque~$\eta'$ est pris suffisamment petit en fonction de~$\ee$, et~$\eta$ en fonction de~$\eta'$,
ces conditions sont satisfaites grâce aux hypothèses~\eqref{cond-thm-pcp-kloomean} ou~\eqref{cond-thm-pcp-klooweil}, respectivement,
et on obtient pour un certain~$\delta>0$
\begin{equation}\label{estim-S1}
\cS_1 = \hat f(0)X_1 + O(x^{1-\delta}KR^{-1})
.\end{equation}

\subsection{Démonstration de la Proposition~\ref{prop-bfi-final}}

Pour compléter l'estimation de~$\cS_1$, il reste à démontrer la Proposition~\ref{prop-bfi-final}. On remarque tout d'abord que la majoration~\eqref{prop-bfi-final-res} n'est non-triviale en pratique que lorsque le majorant est négligeable par rapport à~$MKNLR^{-1}$. On note que cette dernière expression est plus petite d'un facteur~$H$ que la majoration triviale consistant à appliquer l'inégalité triangulaire au membre de gauche de~\eqref{prop-bfi-final-res}. Cela est du à l'utilisation de la formule de Poisson (Lemme~\ref{lemme-poisson}) qui exprime un terme majoré trivialement par~$O(1)$ comme une somme d'exponentielles de taille~$O(H)$.

\bigskip

On note~$\nu:=vd_1d_2$, $k_1 := k$ et
\begin{equation}\label{def-gamma} \gamma_{k_2} := \sum_{n\ell=k_2}\beta_n\lambda_\ell .\end{equation}
L'objet d'étude est
\begin{equation}\label{def-cR}
\cR := \sum_{R<r\leq 2R\\(r,a\nu)=1}\frac 1 r\summ_{d_1k_1\equiv d_2k_2\mod{r}\\(d_1k_1,d_2k_2)=1}u_{k_1}\gamma_{k_2}
\sum_{1\leq|h|\leq H}\hat f\Big(\frac h r\Big)\e\Big(-ah\frac{\bar{\nu k_1}}{r}\Big) .\end{equation}
La suite~$\gamma$ est à support dans~$]NL,4NL]$. On rappelle que~$K$ et~$NL$ ne sont pas nécessairement égaux.

Les entiers~$r, \nu, k_1$ étant deux à deux premiers entre eux, on a la congruence
\[ \nu k_1\bar{\nu k_1}^{(r)}+\nu(d_2k_2-d_1k_1)\bar{\nu d_2k_2}^{(k_1)}+rk_1\bar{rk_1}^{(\nu)}\equiv 1\mod{r\nu k_1} \]
(où~$\bar{a}^{(q)}$ désigne un inverse de~$a$ modulo~$q$). Ainsi, avec~$t:=(d_2k_2-d_1k_1)/r$, on a
\begin{equation}\label{congru-triple}
-ah\frac{\bar{\nu k_1}}r \equiv aht\frac{\bar{\nu d_2k_2}}{k_1}+ah\frac{\bar{rk_1}}{\nu}-\frac{ah}{\nu k_1r}\quad\mod{1}
.\end{equation}
Le dernier terme est~$\ll R^{\ee}|a|\{\nu KM\}^{-1}$. On a donc
\begin{equation}\label{R-Rprime}\begin{aligned}
\cR=&\ \sum_{R<r\leq 2R\\(r,a\nu)=1}\frac 1 r\summ_{d_1k_1\equiv d_2k_2\mod{r}\\(d_1k_1,d_2k_2)=1}u_{k_1}\gamma_{k_2}
\sum_{1\leq|h|\leq H}\hat f\Big(\frac h r\Big)\e\Big(aht\frac{\bar{\nu d_2k_2}}{k_1}+ah\frac{\bar{rk_1}}\nu\Big)
\\ &\ + O(|a|(RKNL)^\ee K\{\nu M\}^{-1}) .\end{aligned}\end{equation}
On note~$\cR'$ le premier terme du membre de droite et pour tout entier~$w$ avec~$0<w<\nu$ et~$(w,\nu)=1$, on note
\[ \phi(w) := t\frac{\bar{\nu d_2k_2}}{k_1}+\frac{\bar{w k_1}}\nu \]
\[ \cR(w) := \sum_{R<r\leq 2R\\(r,a)=1\\r\equiv w\mod \nu}\frac 1 r\summ_{d_1k_1\equiv d_2k_2\mod{r}\\(d_1k_1,d_2k_2)=1}u_{k_1}\gamma_{k_2}
\sum_{1\leq|h|\leq H}\hat f\Big(\frac h r\Big)\e(ah\phi(w)) .\]
Les conditions sur~$r$ sont :
\[ R<r\leq 2R,\qquad(r,a)=1,\qquad r\equiv w\mod{\nu},\qquad d_1k_1\equiv d_2k_2\mod{r}, \]
elles deviennent vis-à-vis de~$t=(d_2k_2-d_1k_1)/r$ :
\[ R<\frac{d_2k_2-d_1k_1}t\leq 2R,\qquad d_2k_2-d_1k_1\equiv w t\mod{\nu t},\qquad (d_2k_2-d_1k_1,at)=t .\]
On détecte la troisième grâce à la relation
\[ \bfUn_{(d_2k_2-d_1k_1,at)=t}=\sum_{\sigma|a\\\sigma t|d_2k_2-d_1k_1}\mu(\sigma). \]
Pour un indice~$\sigma$ de cette somme, on a~$(\sigma,\nu)=1$. Ainsi, pour un certain~$\sigma|a$, on a
\begin{align*}
&\ |\cR'| = \Big|\sum_{0<w<\nu\\(w,\nu)=1}\cR(w)\Big|\\
\leq&\ \tau(|a|)\Bigg|\sum_{0<w<\nu\\(w,\nu)=1}\sum_{t\neq 0}
\summ_{(d_1k_1,d_2k_2)=1\\R<(d_2k_2-d_1k_1)/t\leq2R\\d_2k_2-d_1k_1\equiv w t\mod{\nu t}\\\sigma t|d_2k_2-d_1k_1}
u_{k_1}\gamma_{k_2}\frac t{d_2k_2-d_1k_1}\sum_{1\leq|h|\leq H}\hat f\Big(\frac{ht}{d_2k_2-d_1k_1}\Big)
\e(ah\phi(w))\Bigg|\\
\ll&\ |a|^\ee\Bigg|\sum_{0<w<\nu\\(w,\nu)=1}
\sum_{t\neq0}\summ_{(d_1k_1,d_2k_2)=1\\R<(d_2k_2-d_1k_1)/t\leq2R\\d_2k_2-d_1k_1\equiv w\sigma t\mod{\nu\sigma t}}
u_{k_1}\gamma_{k_2}\frac t{d_2k_2-d_1k_1}\sum_{1\leq|h|\leq H}\hat f\Big(\frac{ht}{d_2k_2-d_1k_1}\Big)
\e(ah\phi(w\sigma))\Bigg|.
\end{align*}
On a~$\sigma t|d_2k_2-d_1k_1$ et~$(d_1k_1,d_2k_2)=1$, donc~$(k_2,\sigma t)=1$. On peut alors utiliser les relations :
\begin{align*}
\bfUn_{d_2k_2\equiv d_1k_1+w\sigma t\mod{\nu\sigma t}} =&\ \frac{\bfUn_{d_2|d_1k_1+w\sigma t}}{\vphi(vd_1\sigma t)} \sum_{\chi\mod{vd_1\sigma t}} \chi(k_2) \bar{\chi((d_1k_1+w\sigma t)/d_2)} \\
\bfUn_{R<(d_2k_2-d_1k_1)/t\leq2R} =&\ \int_{-1/2}^{1/2} \e((d_2k_2-d_1k_1)\vth) F_t(\vth) \dd\vth \end{align*}
où~$F_t(\vth) := \sum_{c\in\bfZ, R<c/t\leq2R} \e(c\vth)\ll\min\{TR,|\vth|^{-1}\}$ pour~$|\vth|\leq 1/2$, avec~$T:=\max\{d_1K, d_2NL\}/R$. On a donc
\begin{align*}
|\cR'|\ll&\ |a|^\ee\Bigg|\sum_{0<w<\nu\\(w,\nu)=1}\sum_{1\leq|t|\leq T}\summ_{(d_1k_1,d_2k_2)=1\\d_2|d_1k_1+w\sigma t}
 \frac 1{\vphi(vd_1\sigma t)} \\
&\quad\times\sum_{\chi\mod{vd_1\sigma t}}\chi(k_2)\bar{\chi((d_1k_1+w\sigma t)/d_2)}u_{k_1}\gamma_{k_2}
\int_{-1/2}^{1/2}\e((d_2k_2-d_1k_1)\vth)F_t(\vth)\dd\vth \\
&\quad\times\sum_{1\leq|h|\leq H}\frac t{d_2k_2-d_1k_1}
\hat f\Big(\frac{ht}{d_2k_2-d_1k_1}\Big)\e(ah\phi(w\sigma))\Bigg|.
\end{align*}
Par ailleurs, en utilisant la définition de~$\hat f$ puis la formule d'inversion de Fourier afin de séparer les variables~$k_1$ et~$k_2$, on a
\[ \frac t{d_2k_2-d_1k_1} \hat f\left(\frac{ht}{d_2k_2-d_1k_1}\right) =
\int_0^{3MR^{-1}}\int_{-\infty}^\infty t\hat f(\eta t) \e(-\eta\xi(d_2k_2-d_1k_1))\e(\xi h)\dd\eta\dd\xi .\]
En injectant la définition~\eqref{def-gamma} et en renommant~$k_1$ en~$k$, on obtient
\begin{align*}
|\cR'|\ll&\ |a|^\ee\sum_{0<w<\nu\\(w,\nu)=1}\sum_{k}|u_k|\sum_{(\ell,k)=1}|\lambda_{\ell}|
\sum_{1\leq|t|\leq T}\int_{-1/2}^{1/2}\int_0^{3MR^{-1}}\int_{-\infty}^\infty|F_t(\vth)t\hat f(\eta t)|\frac 1{\vphi(vd_1\sigma t)} \\
&\times\sum_{\chi\mod{vd_1\sigma t}}\Bigg|\sum_{(n,k)=1}\sum_{1\leq|h|\leq H}\beta_{n}\chi(n)
\e(ah\phi(w\sigma)+d_2n\ell\vth+\xi h-\eta\xi d_2n\ell)
\Bigg|\dd\eta\dd\xi\dd\vth.
\end{align*}
On effectue les changements de variables~$\eta\gets\eta/\ell,\vth\gets\vth/\ell$ :
\begin{align*}
|\cR'|\ll&\ |a|^\ee\int_{-L}^{L}\int_0^{3MR^{-1}}\int_{-\infty}^{\infty}\sup_{1\leq|t|\leq T\\\max\{2|\vth|,L\}\leq \ell\leq 2L}
\Big|\frac1\ell F_t\Big(\frac\vth \ell\Big)\Big|\sup_{1\leq|t|\leq T\\L<\ell\leq 2L}\Big|\frac t \ell\hat f\Big(\frac{\eta t}\ell\Big)\Big| \\
&\times\sum_{0<w<\nu\\(w,\nu)=1}\sum_{k}\sum_{(\ell,k)=1}\sum_{1\leq|t|\leq T}|u_{k}\lambda_{\ell}|\frac1{\vphi(vd_1\sigma t)} \\
&\times\sum_{\chi\mod{vd_1\sigma t}}\Bigg|\sum_{(n,k)=1}\sum_{1\leq|h|\leq H}\beta_{n}\chi(n)
\e\Big(aht\frac{\bar{\nu d_2n\ell}}{k}+ah\frac{\bar{w\sigma k}}{\nu}+d_2n\vth+\xi h-\eta\xi d_2n\Big)
\Bigg|\dd\eta\dd\xi\dd\vth.
\end{align*}
On rappelle que par les hypothèses faites sur~$u, \beta, \lambda$, les sommations sont restreintes aux indices~$k, n, \ell$
tels que~$(k,\nu)=(n\ell,vd_1)=1$.
On permute les sommations sur~$w$ et~$k$ et on effectue le changement de variables~$w\gets w\bar{k\sigma}^{(\nu)}$.
Par ailleurs les deux~$\sup$ sont respectivement~$O(\min\{|\vth|^{-1}, TRL^{-1}\})$
et~$O(\min\{MTL^{-1},|\eta|^{-1},LM^{-1}|\eta|^{-2}\})$. Finalement, pour des entiers~$\sigma, w$ et des réels~$\eta, \xi, \vth$
dépendant au plus de~$v, d_1, d_2$ et~$a$ et vérifiant~$\sigma|a$, $(w, \nu)=1$, et ayant posé
\[ \beta(n,h):=\beta_{n}\e\Big(ah\frac{\bar{w}}{\nu}+d_2n\vth+\xi h-\eta\xi d_2n\Big), \]
on obtient
\begin{equation}\label{eq-lien-cB}
\begin{aligned}
|\cR'|\ll&\ (|a|TR)^\ee\nu MR^{-1}\sum_{k}\sum_{(\ell,k)=1}\sum_{1\leq|t|\leq T}|u_{k}\lambda_{\ell}| \\
&\quad\times\frac1{\vphi(vd_1\sigma t)}\sum_{\chi\mod{vd_1\sigma t}}\Bigg|\sum_{(n,k)=1}\sum_{1\leq|h|\leq H}\beta(n,h)\chi(n)
\e\Big(aht\frac{\bar{\nu d_2n\ell}}{k}\Big)\Bigg| \\
\ll&\ (|a|KTR)^\ee\nu MR^{-1}\{KLT\}^{1/2}\cB^{1/2}
\end{aligned}
\end{equation}
par l'inégalité de Cauchy-Schwarz, avec
\begin{align*}
\cB := \sum_{1\leq|t|\leq T}\frac1{\vphi(vd_1\sigma t)}&\sum_{\chi\mod{vd_1\sigma t}}\sum_{(k,\nu)=1}\Phi_0\Big(\frac k K\Big) \times \\
& \times \sum_{(\ell,k)=1}\Phi_0\Big(\frac \ell L\Big) \Bigg|\sum_{(n,k)=1}\sum_{1\leq|h|\leq H}\beta(n,h)\chi(n)\e\Big(aht\frac{\bar{\nu d_2n\ell}}{k}\Big)\Bigg|^2 .\end{align*}
Ici~$\Phi_0:\bfR_+\to\bfR_+$ dénote une fonction lisse majorée par la fonction indicatrice de l'intervalle~$[1/2, 3]$
et majorant la fonction indicatrice de l'intervalle~$[1,2]$.
En développant le carré et en évaluant la somme sur~$\chi$, on obtient
\[ \cB=\sum_{1\leq|t|\leq T}\sum_{(k,\nu)=1}\sum_{(\ell,k)=1}\Phi_0\Big(\frac k K\Big)\Phi_0\Big(\frac \ell L\Big)
\summmm_{1\leq|h|,|h'|\leq H\\(nn',vd_1\sigma tk)=1\\n\equiv n'\mod{vd_1\sigma t}}
\beta(n,h)\bar{\beta(n',h')}\e\Big(at(n'h-nh')\frac{\bar{\nu d_2 nn'\ell}}{k}\Big) .\]
On pose pour tous entiers~$e, q$ :
\[ B_{e, q}:=\sum_{n, n'\\nn'=q}\sum_{1\leq|t|\leq T, t|e\\n\equiv n'\mod{vd_1\sigma t}\\(nn',vd_1\sigma t)=1}
\summ_{1\leq|h|,|h'|\leq H\\n'h-nh'=e/t}\beta(n,h)\bar{\beta(n',h')} .\]
On a~$B_{e, q}=0$ si~$e$ et~$q$ ne vérifient pas~$|e|\leq 4HNT$ et~$N^2<q\leq4N^2$. Par ailleurs,
\[ \cB=\sum_{e}\sum_{q}B_{e,q}\underset{(k, \nu q \ell)=1}{\sum_{k}\sum_\ell}\Phi_0\Big(\frac k K\Big)\Phi_0\Big(\frac \ell L\Big)
\e\Big(ae\frac{\bar{\nu d_2q \ell}}{k}\Big) .\]
On sépare la contribution des termes avec~$e=0$ :
\begin{equation}\label{B-Be-Bne} \cB = \cB(e=0)+\cB(e\neq 0) \end{equation}
avec
\[ \cB(e=0)\ll KL\sum_{n}\sum_{n'}|\beta_{n}\beta_{n'}|\sum_{1\leq|t|\leq T\\n\equiv n'\mod t}
\summ_{1\leq|h|,|h'|\leq H\\nh'=n'h}1 .\]
La contribution des termes de la somme avec~$n=n'$ est~$O(KNLHT)$. Le reste contribue
\[ \ll KL\sum_{n}|\beta_{n}|\sum_{\delta|n}\sum_{(n',n)=\delta}|\beta_{n'}|\tau(|n-n'|)\sum_{1\leq|h|\leq H\\(n/\delta)|h}1
\ll (HN)^\ee KLNH .\]
On a donc
\begin{equation}\label{majo-Be} \cB(e=0)\ll (HN)^\ee KLNHT.\end{equation}
Le Lemme~\ref{kloo-mean} s'applique à la somme~$\cB(e\neq0)$ avec
\[ C\gets K, \quad D\gets L, \quad N\gets 4|a|HNT, \quad R\gets 4\nu d_2N^2, \quad S\gets 1 \]
et permet d'écrire
\begin{equation}\label{Bne-CS}\begin{aligned}
\cB(e\neq 0) \ll&\ (|a|\nu KLNR)^{\ee}\Big\{K(\nu d_2N^2+|a|HNT)(K+\nu d_2LN^2)\\
&+K^2L\sqrt{(\nu d_2N^2+|a|HNT)N^2}+\nu d_2|a|L^2HTN^3\Big\}^{1/2}\Big\{\sum_{e\neq0}\sum_{q}|B_{e,q}|^2\Big\}^{1/2}
.\end{aligned}\end{equation}
Afin d'estimer le dernier crochet, on sépare la contribution des termes avec~$n=n'$ : on écrit
\begin{equation}\label{decomp-Beq} |B_{e,q}| \ll B_1(|e|,q)+B_2(|e|,q) \end{equation}
avec
\[ B_1(e,q) := \begin{cases}
|\beta_{n}|^2\card\{(t, h, h') | 1\leq|h|,|h'|\leq H, 1\leq|t|\leq T\ |\ tn(h-h')=e\}&\text{ si }q=n^2, n|e\\
0&\text{ sinon},
\end{cases} \]
\[ B_2(e,q):=\sum_{nn'=q\\n\neq n'}|\beta_{n}\beta_{n'}|\sum_{1\leq t\leq T, t|e\\n\equiv n'\mod t}
\summ_{1\leq|h|,|h'|\leq H\\n'h-nh'=e/t}1 .\]
On a
\begin{equation}\label{majo-B1}\begin{aligned}
&\ \sum_{n}\sum_{e'}|B_1(ne',n^2)|^2 \\
\ll&\ \sum_{n}|\beta_{n}|^4\sum_{0<e'\leq2HT}\Big\{\sum_{t|e'}\card\{1\leq |h|,|h'|\leq H\ |\ h-h'=e'/t\}\Big\}^2\\
\ll&\ (HT)^\ee H^3NT
.\end{aligned}\end{equation}
En ce qui concerne~$B_2(e,q)$, on a
\[ B_2(e, q) \leq \sum_{nn'=q\\n\neq n'}|\beta_{n}\beta_{n'}|\sum_{1\leq t\leq T\\t|e}\sum_{1\leq|h|,|h'|\leq H\\n'h-nh'=e/t}1
\ll \tau(e)\Big(1+\frac H N\Big)\sum_{nn'=q}(n,n')|\beta_{n}\beta_{n'}| .\]
On a donc
\begin{align*}
\sum_{e}\sum_{q}B_2(e,q)^2\ll&\ (HNT)^{\ee/2}(H+N)N^{-1}\sum_{e}\sum_{q}\sum_{nn'=q}(n,n')|\beta_{n}\beta_{n'}| B_2(e,q) \\
\ll&\ (HNT)^{\ee/2}(H+N)H^2N^{-1}\sum_{n_1}\sum_{n_2}(n_1,n_2)|\beta_{n_1}\beta_{n_2}|\sum_{n_3n_4=n_1n_2\\n_3\neq n_4}|\beta_{n_3}\beta_{n_4}|\tau(|n_3-n_4|) \\
\ll&\ (HNT)^{\ee}(H+N)H^2N^{-1}\summm_{n_1n_2=n_3n_4}(n_1,n_2)|\beta_{n_1}\beta_{n_2}\beta_{n_3}\beta_{n_4}|.
\end{align*}
On note que
\[ \summm_{n_1n_2=n_3n_4}(n_1,n_2)|\beta_{n_1}\beta_{n_2}\beta_{n_3}\beta_{n_4}|\ll N^2(\log N)^5 .\]
En regroupant cette dernière estimation avec~\eqref{majo-B1}, \eqref{decomp-Beq}, \eqref{Bne-CS}, \eqref{majo-Be}, \eqref{B-Be-Bne}, \eqref{eq-lien-cB} et~\eqref{R-Rprime}, et grâce aux hypothèses~\eqref{cond-prop-kloomean}, on obtient 
\begin{align*}
\cR \ll&\ |a|(KNLR)^{\ee}KM^{-1} + (KNLR)^{5\ee} MR^{-1}\{K^2LR^{-1}\}^{1/2}\Big\{K^2NLM^{-1} \\
&\quad+\{KN^4L+K^2N^2L+KN^3L^2M^{-1}\}^{1/2}RNM^{-1}\Big\}^{1/2}
.\end{align*}
En utilisant~$M\leq R\leq K$, $NL\leq KM$ et~$K\leq NM$, on obtient
\begin{align*}
 \cR(KNL)^{-10\ee} \ll&\ KM^{-1} + M^{1/2}K^{5/4}N^{1/2}L^{3/4}R^{-3/2}\{K^{3/4}L^{1/4}+R^{1/2}N+R^{1/2}K^{1/4}N^{1/2} \}
.\end{align*}
Sous les conditions~\eqref{cond-prop-kloomean}, le membre de droite est~$O(MKNL\cE^{-1})$ et on obtient finalement la majoration voulue
\[ \cR \ll (KNL)^{10\ee} MKNL\cE^{-1} .\]

\bigskip

Dans un second temps, on majore la quantité~$\cR$ en utilisant le Lemme~\ref{majo-kloo-complete}.
On reprend l'étude précédente sans majorer trivialement la contribution du terme~$-ah/(\nu k_1r)$
provenant de l'équation~\eqref{congru-triple}, et en utilisant l'égalité modulo~$1$
\[ \frac{\bar{\nu d_2\ell n}}k\equiv -\frac{\bar k}{\nu d_2n\ell}+\frac1{\nu d_2n\ell k}\quad \mod{1} .\]
Par analogie avec la première majoration dans~\eqref{eq-lien-cB}, on obtient pour un certain~$\sigma|a$ et trois réels~$\xi, \eta, \vth$ dépendant au plus de~$v, d_1, d_2$ la majoration
\begin{align*}
|\cR|\ll&\ (|a|TR)^\ee\nu MR^{-1}\sum_{1\leq|t|\leq T}
\sum_k\sum_{(\ell,k)=1}\frac1{\vphi(vd_1\sigma t)}\sum_{\chi\mod{vd_1\sigma t}}|u_k\lambda_\ell| \\
&\qquad\times\sum_{1\leq|h|\leq H}\Bigg|\sum_{(n,k)=1}\beta_n\chi(n)\e\Big(-\frac{aht\bar k}{\nu d_2n\ell}
+\frac{aht}{\nu d_2n\ell k}+\eta\xi d_2n\ell+d_2 n\ell\vth\Big)\Bigg|.
\end{align*}
où l'on rappelle que~$\cR$, défini par~\eqref{def-cR}, est l'objet que l'on souhaite majorer. On rappelle également que~$T:=\max\{d_1K, d_2NL\}/R$. L'inégalité de Cauchy--Schwarz fournit
\begin{equation}\label{R-CS} |\cR|\ll(|a|TR)^\ee\nu MR^{-1}\big\{KLHT\big\}^{1/2}\cD_0^{1/2} \end{equation}
où
\begin{align*}
\cD_0 := &\sum_{K<k\leq4K\\(k, \nu d_2)=1}\sum_{L<\ell\leq2L\\(\ell, k)=1}\sum_{1\leq|t|\leq T}\frac1{\vphi(vd_1\sigma t)} \\
&\times\sum_{\chi\mod{vd_1\sigma t}} \sum_{1\leq|h|\leq H}\Bigg|\sum_{(n,k)=1}\beta_n\chi(n)
\e\Big(-\frac{aht\bar k}{\nu d_2n\ell}+\frac{aht}{\nu d_2n\ell k}+\eta\xi d_2n\ell+d_2n\ell\vth\Big)\Bigg|^2
.\end{align*}
En développant le carré, il vient
\begin{equation}\label{majo-cD}\begin{aligned}\cD_0 \leq&\ \sum_{1\leq|h|\leq H}\summ_{N<n\leq2N\\N<n'\leq2N}
\summ_{1\leq|t|\leq T,\ L<\ell\leq2L\\n \equiv n'\mod{vd_1\sigma t}}
\Bigg|\sum_{K<k\leq 2K\\(k,\nu d_2nn'\ell)=1}\e\Big(\frac{aht\bar k}{\nu d_2\ell[n, n']}\frac{n-n'}{(n, n')}+\frac{aht(n'-n)}{\nu d_2nn'\ell k}\Big)\Bigg|.
\end{aligned}\end{equation}
On élimine le second terme dans l'exponentielle en intégrant par parties. On écrit le membre de droite de cette égalité sous la forme :
\[ \sum_{z\in\cZ}\Big|\sum_{K<k\leq2K}g_1(z,k)g_2(z,k)\Big| \]
où~$\cZ$ est l'ensemble des quintuplets d'entiers~$(h,n,n',\ell,t)$ correspondant à des indices dans la somme
du membre de droite de~\eqref{majo-cD}, et les fonctions~$g_1$ et~$g_2$ sont définies par
\[ g_1(z, k) := \e\Big(\frac{aht\bar k}{\nu d_2\ell[n, n']}\frac{n-n'}{(n, n')}\Big),
\qquad g_2(z, k) := \e\Big(\frac{aht(n'-n)}{\nu d_2nn'\ell k}\Big) \qquad (z=(h,n,n',\ell,t)\in\cZ).\]
En notant~$G_1(z, k) := \sum_{K<k'\leq k}g_1(z, k')$, le membre de droite de~\eqref{majo-cD} vaut
\begin{align*}
\sum_{z\in\cZ}\Big|\int_{\xi=K+}^{2K+}g_2(z, \xi)\dd G_1(z, \xi)\Big| \leq&\ \sup_{z\in\cZ\\K<\xi\leq2K}|g_2(z,\xi)|\times\sum_{z\in\cZ}\big|G_1(z,2K)\big| \\
&\ +K\sup_{z\in\cZ\\K<\xi\leq2K}\Big|\frac{\partial g_2}{\partial \xi}(z, \xi)\Big|\times\sup_{K<\xi\leq2K}\sum_{z\in\cZ}\big|G_1(z,\xi)\big|.
\end{align*}
On a donc
\begin{equation}\label{D0-DKp} \cD_0 \ll \Big(1+\frac{|a|HT}{\nu d_2KNL}\Big)\sup_{K<K'\leq2K}{\cD(K')} \end{equation}
où on a noté
\[ \cD(K') := \sum_{1\leq|h|\leq H}\summ_{N<n\leq2N\\N<n'\leq2N}\sum_{1\leq|t|\leq T\\n \equiv n'\mod{vd_1\sigma t}}\sum_{L<\ell\leq2L}
\Bigg|\sum_{K<k\leq K'\\(k,\nu d_2nn'\ell)=1}\e\Big(\frac{aht\bar k}{\nu d_2\ell[n, n']}\frac{n-n'}{(n, n')}\Big)\Bigg|. \]
On remarque que~$|a|HT\ll\nu d_2KNL$ par hypothèse.

La contribution des indices~$n, n'$ avec~$n=n'$ est~$O(HTKNL)$. La majoration~\eqref{somme-exp-weil} fournit
\begin{equation}\label{DKp-D1-D2} \cD(K') \ll HTKNL + \cD_1 + (KNL)^\ee\nu d_2\cD_2 \end{equation}
avec
\[ \cD_1 := K\sum_{1\leq|h|\leq H}\summ_{N<n'<n\leq2N\\n\equiv n'\mod{vd_1\sigma t}}\sum_{1\leq|t|\leq T}\sum_{L<\ell\leq 2L}
\frac{(\ell[n,n'],aht(n-n')/(n,n'))}{\ell[n,n']}, \]
\[ \cD_2 := \sum_{1\leq|h|\leq H}\sum_{1\leq|t|\leq T}\summ_{N<n'<n\leq2N\\n\equiv n'\mod{vd_1\sigma t}}\sum_{L<\ell\leq 2L}
\big(\ell[n,n']\big)^{1/2}\Big(\ell[n,n'],aht\frac{n-n'}{(n,n')}\Big)^{1/2}. \]
Soit~$\eta>0$. Dans~$\cD_1$, on sépare les sommants suivant la valeur de~$d=(n, n')$ et on évalue les sommes sur~$h$ puis sur~$\ell$. On obtient
\[ \cD_1 \ll (RKNL|a|)^{\eta}KHN^{-2}\sum_{d\leq 2N}d\sum_{1\leq|t|\leq T}\summ_{N/d<n'<n\leq2N/d\\(n,n')=1\\dn\equiv dn'\mod{vd_1\sigma t}}
(dnn',at(n-n')) .\]
Dans cette somme, on a~$(nn', n-n')=1$. Le terme général de la dernière somme est donc inférieur à~$d(nn',at)$.
En séparant de nouveau suivant la valeur de~$\delta=(t, d)$, on obtient
\[ \cD_1 \ll (RKNL|a|)^{\eta}KHN^{-2}\sum_{\delta\leq T}\delta^2\sum_{d\leq 2N/\delta}d^2\sum_{1\leq|t|\leq T/\delta\\(t,d)=1}
\summ_{N/(d\delta)<n'<n\leq2N/(d\delta)\\(n,n')=1\\dn\equiv dn'\mod{vd_1\sigma t}}
(nn',at\delta) .\]
La condition de congruence sur~$n, n'$ implique~$t|d(n-n')$ donc~$t|n-n'$. Puisque~$(n,n')=1$, on a~$(t,nn')=1$, ainsi en permutant et en évaluant d'abord la somme en~$t$ par~$O(K^\ee)$, on obtient
\begin{equation}\label{majo-D1}\begin{aligned}\cD_1&\ \ll (RK^2NL|a|)^{\eta}KHN^{-2}\sum_{\delta\leq T}\delta^2\sum_{d\leq 2N/\delta}d^2\Big(\sum_{n\leq2N/(d\delta)}(n,a\delta)\Big)^2\\&\ \ll (RK^2N^2|a|^2)^{\eta}KHN .\end{aligned}\end{equation}
Dans~$\cD_2$, on procède de même. En séparant les sommants suivant la valeur de~$d=(n,n')$
et en évaluant les sommes sur~$h$ et~$\ell$, on obtient
\[ \cD_2\ll (RKNL|a|)^\eta HNL^{3/2}\sum_{d\leq2N}d^{-1/2}\sum_{1\leq|t|\leq T}\summ_{N/d<n'<n\leq2N/d\\(n,n')=1\\dn\equiv dn'\mod{vd_1\sigma t}}
\big(dnn',at(n-n')\big)^{1/2} .\]
En séparant suivant la valeur de~$\delta=(t, d)$, on obtient
\[ \cD_2\ll (RKNL|a|)^\eta HNL^{3/2}\sum_{\delta\leq T}\delta^{-1/2}\sum_{d\leq2N/\delta}d^{-1/2}\sum_{1\leq|t|\leq T/\delta\\(t,d)=1}
\summ_{N/(d\delta)<n'<n\leq2N/(d\delta)\\(n,n')=1\\dn\equiv dn'\mod{vd_1\sigma t}} \big(dnn',at(n-n')\big)^{1/2} .\]
Le terme général de la dernière somme est inférieur à~$d^{1/2}(nn',at(n-n'))^{1/2} = d^{1/2}(nn',at)^{1/2} = d^{1/2}(nn',a)^{1/2}$, ainsi
\begin{equation}\label{majo-D2}\begin{aligned}\cD_2&\ \ll (RK^2NL|a|)^{\eta} HNL^{3/2}\sum_{\delta\leq T}\delta^{-1/2}\sum_{d\leq2N/\delta}
\Big(\sum_{n\leq2N/(d\delta)}(n,a)^{1/2}\Big)^2\\&\ \ll (RK^2N^2L|a|^2)^{\eta}HL^{3/2}N^3 .\end{aligned}\end{equation}

Avec~$\eta=\ee/2$, on injecte les majorations~\eqref{majo-D1} et~\eqref{majo-D2} dans~\eqref{DKp-D1-D2} puis~\eqref{D0-DKp} pour obtenir
\[ \cD_0\ll (RKNL|a|)^{\ee}HTKNL+(KR|a|)^\ee\nu d_2HL^{3/2}N^3 \]
et finalement, grâce à~\eqref{R-CS},
\[ |\cR|\ll (RKNL|a|)^{3\ee}\big\{TKLN^{1/2}+T^{1/2}K^{1/2}N^{3/2}L^{5/4}\big\} .\]
Par définition, $T\leq d_1KR^{-1}$, les conditions~\eqref{cond-prop-klooweil} impliquent donc
\[ TKLN^{1/2}\leq d_1 MKNLR^{-1}\cE^{-1}, \qquad T^{1/2}K^{1/2}N^{3/2}L^{5/4}\leq d_1^{1/2}MKNLR^{-1}\cE^{-1} \]
ce qui implique la majoration voulue
\[ |\cR|\ll (MKNL)^{10\ee}MKNLR^{-1}\cE^{-1} .\]

\subsection{Contribution des termes principaux}\label{section-tp}

Les calculs des sections précédentes et plus particulièrement les équations~\eqref{majo-delta-init}, \eqref{dispersion}, \eqref{estim-S3}, \eqref{estim-S2} et~\eqref{estim-S1} montrent que le membre de gauche de~\eqref{estim-MNL} est majoré par
\begin{equation}\label{fin-dispersion}
\Delta(M, N, L, R) \ll \Big( \hat f(0) MR\Big\{X_3-2\Re X_2 + X_1\Big\} \Bigg)^{1/2} + O(x^{1-\delta}) \end{equation}
pour un certain~$\delta>0$. On remarque que
\begin{align*}
X_1-2\Re X_2+X_3=&\ \sum_{R<r\leq 2R\\(r,a_1a_2)=1}\frac 1 r\sum_{0<b<r\\(b,r)=1}
\Bigg|\sum_{k\equiv b\mod{r}}u_k-\frac 1{\vphi(r)}\sum_{(k,r)=1}u_k\omega_\ee(k\bar b ; r)\Bigg|^2 \\
\leq&\ \frac 1 R\sum_{R<r\leq2R}\frac 1{\vphi(r)}\sum_{\chi\prim\\\cond(\chi)>x^\ee\\\cond(\chi)|r}\Big|\sum_{(k,r)=1}u_k\chi(k)\Big|^2.
\end{align*}
Des calculs similaires à~\cite[Section~4]{HarperBV2012} fournissent
\begin{equation*}\begin{aligned}
X_1-2\Re X_2+X_3\leq&\ \frac 1 R\sum_{x^\ee<s\leq 2R}\sum_{\chi\mod{s}\\\chi\prim}\sum_{r\leq 2R\\s|r}\frac 1{\vphi(r)}
\Big|\sum_{(k,r/s)=1}u_k\chi(k)\Big|^2 \\
\leq &\ \frac 1 R\sum_{x^\ee<s\leq2R}\sum_{\chi\mod{s}\\\chi\prim}\sum_{r\leq 2R\\s|r}\frac{\tau(r/s)}{\vphi(r)}
\sum_{d|r/s}\Big|\sum_{d|k}u_k\chi(k)\Big|^2 \\
\ll &\ \frac{(\log R)^2}R \sum_{d\leq R}\frac{\tau(d)}{\vphi(d)}\sum_{x^\ee<s\leq2R}\frac1{\vphi(s)}
\sum_{\chi\mod{s}\\\chi\prim}\Big|\sum_{k'}u_{dk'}\chi(k')\Big|^2 \\
= &\ \frac{(\log R)^2}R \sum_{d\leq R}\frac{\tau(d)}{\vphi(d)}\int_{x^\ee}^\infty\sum_{x^\ee<s\leq\min\{2R,t\}}\frac s{\vphi(s)}
\sum_{\chi\mod{s}\\\chi\prim}\Big|\sum_{k'}u_{dk'}\chi(k')\Big|^2\frac{\dd t}{t^2} \\
\ll &\ \frac{(\log R)^2}R \sum_{d\leq R}\frac{\tau(d)}{\vphi(d)}\Big(\frac K{dx^\ee}+R\Big)\frac K d \\
\ll &\ K^2R^{-1}x^{-\ee/2}.
\end{aligned}\end{equation*}
Dans l'avant-dernière inégalité, on a fait usage du grand crible sous la forme du Lemme~\ref{gd-crible}.
En injectant cela dans~\eqref{dispersion} et en utilisant~$\hat f(0)=O(M)$ et~$KM=x$, on obtient pour un certain réel positif~$\delta$,
\[ \Delta(M,N,L,R)\ll x^{1-\delta} .\]
Cela démontre le Théorème~\ref{thm-BFI-MNL}.

\subsection{Fin de la démonstration du Théorème~\ref{thm-BFI-Sxy}}

Lorsque~$y$ n'est pas trop proche de~$x$, la fonction caractéristique de l'ensemble~$S(x,y)$ est facilement approchée par
une combinaison linéaire de convolutions. Cette propriété de factorisation découle du lemme suivant, {\it cf.}~\cite[lemme~3.2]{FT96}.
\begin{lemme}\label{lemme-facto}
Étant donné $y\geq 2$, $N_1, N_2 \geq 1$, il existe pour tout~$n>yN_1N_2$ avec~$P^+(n)\leq y$ une unique factorisation~$n = n_0 n_1 n_2$
telle que
\[ N_2 < n_2 \leq N_2 P^-(n_2), \qquad P^+(n_2) \leq y, \]
\[ N_1 < n_1 \leq N_1 P^-(n_1), \qquad P^+(n_1) \leq P^-(n_2), \]
\[ P^+(n_0) \leq P^-(n_1) .\]
\end{lemme}
\begin{prop}\label{prop-BFI-Sxy-dyadique}
Soit~$\ee>0$. Il existe des réels~$c, \delta>0$ tels que lorsque~$a_1, a_2\in\bfZ\setminus\{0\}$, $(a_1, a_2)=1$ et~$x, y\in\bfR$
avec~$(\log x)^c\leq y\leq x^{1/c}$, on ait
\begin{equation}\label{dyad-estim}\sum_{r\leq x^{3/5-4\ee}\\(r, a_1a_2)=1}\Big|\sum_{x<n\leq2x\\P^+(n)\leq y\\n\equiv a_1\bar{a_2}\mod{r}}1
-\frac1{\vphi(r)}\sum_{x<n\leq2x\\P^+(n)\leq y\\(n, r)=1}\omega_\ee(n\bar{a_1}a_2 ; r)\Big| \ll x^{1-\delta} \qquad (|a_1|, |a_2|\leq x^\delta), \end{equation}
\begin{equation}\label{dyad-estim-unif}\sum_{r\leq x^{6/11-5\ee}\\(r, a_1a_2)=1}\Big|\sum_{x<n\leq2x\\P^+(n)\leq y\\n\equiv a_1\bar{a_2}\mod{r}}1
-\frac1{\vphi(r)}\sum_{x<n\leq2x\\P^+(n)\leq y\\(n, r)=1}\omega_\ee(n\bar{a_1}a_2 ; r)\Big| \ll x^{1-\delta} \qquad (|a_1|\leq x^{1-\ee},|a_2|\leq x^\delta). \end{equation}
\end{prop}
On rappelle que~$\omega_\ee(k ; r)$ est défini par~\eqref{def-omega}.
\begin{proof}
On montre dans un premier temps l'estimation~\eqref{dyad-estim}. Soit~$\ee>0$ et~$c>1/\ee$. On suppose~$(\log x)^c\leq y\leq x^{1/c}$
et on définit pour tout~$a\mod{r}$, $(a, r)=1$,
\[ E_\ee(n;a,r) := \bfUn_{n \equiv a\mod{r}} - 1/\vphi(r) \bfUn_{(n, r)=1}\omega_\ee(n\bar{a} ; r) .\]
Pour tout~$R\in\bfR$ avec~$x^{4/9}\leq R\leq x^{3/5-4\ee}$, on pose
\[ M_0 := x^{1/2+2\ee}R^{-1/6}, \quad L_0:=x^{1/2-\ee}R^{-1/2}, \quad N_0:=x^{-\ee}R^{2/3}. \]
On a~$yM_0N_0<x$, donc d'après le Lemme~\ref{lemme-facto},
\begin{align*}&\ \sum_{R<r\leq2R\\(r, a_1a_2)=1}\Big|\sum_{x<n\leq2x\\P^+(n)\leq y}E_\ee(n;a_1\bar{a_2}, r)\Big|
\\=&\ \sum_{R<r\leq2R\\(r, a_1a_2)=1}\Big|\sum_{L_0<\ell\leq L_0P^-(\ell)\\P^+(\ell)\leq y}\sum_{M_0<m\leq M_0P^-(m)\\P^+(m)\leq P^-(\ell)}
\sum_{x<mn\ell\leq2x\\P^+(n)\leq P^-(m)}E_\ee(mn\ell;a_1\bar{a_2}, r)\Big| .\end{align*}
Pour tous~$M, N, L$ tels que
\[ M_0\leq M\leq yM_0/2, \qquad L_0\leq L\leq yL_0/2, \qquad y^{-2}N_0\leq N\leq N_0, \]
on considère la somme
\begin{align*} \Delta^*(M, N, L, R) := \sum_{R<r\leq2R\\(r, a_1a_2)=1}\Bigg|
\sum_{L<\ell\leq2L\\\ell\leq L_0P^-(\ell)\\P^+(\ell)\leq y}
\sum_{M<m\leq2M\\m\leq M_0P^-(m)}&\sum_{N<n\leq2N}E_\ee(mn\ell;a_1\bar{a_2}, r)\times\\&\times\bfUn_{x<mn\ell\leq 2x}\bfUn_{P^+(n)\leq P^-(m)}\bfUn_{P^+(m)\leq P^-(\ell)}\Bigg| .\end{align*}
On a
\[ \sum_{R<r\leq2R\\(r, a_1a_2)=1}\Big|\sum_{x<n\leq2x\\P^+(n)\leq y}E_\ee(n;a_1\bar{a_2}, r)\Big|
\leq \sum_{i=0}^{\floor{\log y/\log 2}}\sum_{j=0}^{\floor{\log y/\log 2}}\sum_{k=0}^{\floor{2\log y/\log 2}}\Delta^*(M_02^i, N_02^{-k}, L_02^j, R) .\]

On remarque que pour tout~$\vth\in\bfR\setminus\{0\}$ et~$T\geq 1$ on a
\begin{equation}\label{perron-like} \bfUn_{\vth>0} = \frac12+\frac1{2\pi i}\int_{1/T\leq|t|\leq T}\frac{\e^{i\vth t}}t\dd t
+O\Big(\frac{|\vth|+|\vth|^{-1}}T\Big) .\end{equation}
En effet, le membre de gauche vaut
\[\frac12+\int_0^\infty\frac{\sin(\vth t)}{\pi t}\dd t \]
et on a les inégalités
\[ \Big| \int_0^{1/T}\frac{\sin(\vth t)}{\pi t}\dd t \Big| \leq \frac{|\vth|}{\pi T}\quad\text{et}\quad\Big| \int_T^\infty\frac{\sin(\vth t)}{\pi t}\dd t\Big|
\leq \frac{|\cos(\vth T)|}{\pi |\vth|T}+\int_T^\infty\frac{|\cos(\vth t)|}{\pi |\vth|t^2}\dd t
\leq \frac2{\pi|\vth|T}.\]
Dans l'expression définissant~$\Delta^*(M,N,L,R)$, on applique l'estimation~\eqref{perron-like} avec~$T=x^2$ quatre fois, pour les paramètres
\begin{align*} \vth \in
\{&\log(\floor{2x}+1/2)-\log(mn\ell),\quad \log(mn\ell)-\log(\floor{x}+1/2),\\
&P^-(m)+1/2-P^+(n),\qquad\quad P^-(\ell)+1/2-P^+(m)\}. \end{align*}
On vérifie que pour chaque tel~$\vth$, on a~$\max\{|\vth|, |\vth|^{-1}\}\ll y^2x$. On obtient donc
\begin{equation}\label{delta-T}\begin{aligned}
\Delta^*(M, N, L, R) = \sum_{R<r\leq2R\\(r, a_1a_2)=1}&\Big|\sum_{L<\ell\leq2L\\\ell\leq L_0P^-(\ell)\\P^+(\ell)\leq y}\sum_{M<m\leq2M\\m\leq M_0P^-(m)}\sum_{N<n\leq2N}
T(m, n, \ell, x) E_\ee(mn\ell;a_1\bar{a_2}, r) \Big|
\\&\quad + O\Big(\frac{y^2(\log x)^3}{x}\sum_{R<r\leq2R\\(r,a_1a_2)=1}\sum_{n\leq2x}\tau_3(n)|E_\ee(n;a_1\bar{a_2}, r)|\Big)
\end{aligned}\end{equation}
avec, en notant~$\cI:=[-x^2, -x^{-2}]\cup[x^{-2}, x^2]$,
\begin{align*}
T(m,n,\ell,x) := &\ 
\Big(\frac12+\frac1{2\pi i}\int_\cI(\floor{2x}+1/2)^{it}(mn\ell)^{-it}\frac{\dd t}t\Big)
\Big(\frac12+\frac1{2\pi i}\int_\cI(mn\ell)^{it}(\floor{x}+1/2)^{-it}\frac{\dd t}t\Big)\\\times&\ 
\Big(\frac12+\frac1{2\pi i}\int_\cI\e^{(P^-(m)+1/2)it}\e^{-P^+(n)it}\frac{\dd t}t\Big)
\Big(\frac12+\frac1{2\pi i}\int_\cI\e^{(P^-(\ell)+1/2)it}\e^{-P^+(m)it}\frac{\dd t}t\Big),
\end{align*}
et où~$\tau_3(n)$ désigne le nombre de représentations de~$n$ en produits de 3 entiers. On a
\[ \sum_{R<r\leq2R\\(r,a_1a_2)=1}\sum_{n\leq2x}\tau_3(n)|E_\ee(n;a_1\bar{a_2}, r)| \ll R+\sum_{n\leq2x\\a_2n\neq a_1}\tau_3(n)\tau(|a_2n-a_1|)
+(\log R)x^{2\ee}\sum_{n\leq2x}\tau_3(n)\ll x^{1+3\ee} .\]
On développe le terme~$T(m, n, \ell, x)$. Dans le but de simplifier la forme de l'expression obtenue, on note~$K := i\pi/(8\log x)$ de sorte que
\[ \frac12 = \frac1{2\pi i}\int_\cI\frac{K\dd t}t .\]
On obtient
\[ T(m,n,\ell,x) = \sum_{j=0}^{15}\frac1{(2\pi)^4}\iiiint_{\cI^4}K^{e_j}\alpha^{(j)}_m\beta^{(j)}_n\lambda^{(j)}_\ell\gamma^{(j)}_x
\frac{\dd t_1\dd t_2\dd t_3\dd t_4}{t_1t_2t_3t_4} \]
pour certains nombres complexes~$\alpha^{(j)}_m, \beta^{(j)}_n, \lambda^{(j)}_\ell, \gamma^{(j)}_x$ de modules~$1$, pouvant dépendre des~$t_k$,
et certains entiers positifs ou nuls~$e_j$. En injectant cela dans l'expression~\eqref{delta-T}, on obtient
\[
\Delta^*(M,N,L,R) \ll \sum_{j=0}^{15}\iiiint_{\cI^4}\Big\{\cdots\Big\}\frac{\dd t_1\dd t_2\dd t_3\dd t_4}{|t_1t_2t_3t_4|} + x^{6\ee}.
\]
\[ \Big\{\cdots\Big\} = \sum_{R<r\leq2R\\(r,a_1a_2)=1}
\Big|\sum_{L<\ell\leq2L\\\ell\leq L_0P^-(\ell)\\P^+(\ell)\leq y}\sum_{M<m\leq2M\\m\leq M_0P^-(m)}\sum_{N<n\leq2N}
\alpha^{(j)}_m\beta^{(j)}_n\lambda^{(j)}_\ell E_\ee(mn\ell;a_1\bar{a_2}, r)\Big| \]
Par construction de~$M_0, N_0, L_0$, les conditions~\eqref{cond-thm-pcp-kloomean} sont satisfaites vis-à-vis de~$M, N, L, R$,
il existe donc~$\delta>0$ tel que pour tout~$j\in\{0, \ldots, 15\}$ et uniformément lorsque~$(t_1, \ldots, t_4)\in\cI^4$, on~ait
\[ \sum_{R<r\leq2R\\(r,a_1a_2)=1}\Big|\sum_{L<\ell\leq2L\\\ell\leq L_0P^-(\ell)\\P^+(\ell)\leq y}\sum_{M<m\leq2M\\m\leq M_0 P^-(m)}\sum_{N<n\leq2N}
\alpha^{(j)}_m\beta^{(j)}_n\lambda^{(j)}_\ell E_\ee(mn\ell;a_1\bar{a_2}, r)\Big| \ll x^{1-\delta} .\]
On obtient~$\Delta^*(M, N, L, R) \ll (\log x)^4x^{1-\delta}$, puis
\[ \sum_{R<r\leq2R\\(r, a_1a_2)=1}\Big|\sum_{x<n\leq2x\\P^+(n)\leq y}E_\ee(n;a_1\bar{a_2}, r)\Big| \ll (\log y)^3(\log x)^4x^{1-\delta}.\]
En sommant pour~$R=x^{3/5-4\ee}2^{-j}$, $j\in\{1, \ldots, \floor{(1/10-3\ee)\log x/\log 2}+1\}$, on obtient
\[ \sum_{x^{1/2-\ee}<r\leq x^{3/5-4\ee}\\(r, a_1a_2)=1}\Big|\sum_{x<n\leq2x\\P^+(n)\leq y}E_\ee(n;a_1\bar{a_2}, r)\Big| \ll (\log y)^3(\log x)^5x^{1-\delta}\ll x^{1-\delta/2} .\]
Il découle par ailleurs du Lemme~\ref{lemme-harper} que
\[ \sum_{r\leq x^{1/2-\ee}\\(r, a_1a_2)=1}\Big|\sum_{x<n\leq2x\\P^+(n)\leq y}E_\ee(n;a_1\bar{a_2}, r)\Big| \ll x^{1-\delta/2} \]
quitte à diminer la valeur de~$\delta$ et augmenter celle de~$c$. Ceci prouve la Proposition~\ref{prop-BFI-Sxy-dyadique}.

L'estimation~\eqref{dyad-estim-unif} se montre par une méthode similaire, pour les choix des paramètres suivants :
\[ x^{4/9}\leq R\leq x^{6/11-5\ee}, \quad M_0:=x^{1/3+2\ee}R^{2/9}, \quad L_0:=R^{2/3}x^{-\ee}, \quad N_0:=x^{2/3-\ee}R^{-8/9}. \]
\end{proof}

\begin{proof}[Démonstration du Théorème~\ref{thm-BFI-Sxy}]
Soit~$\ee>0$ et~$c, \delta$ les réels donnés par la Proposition~\ref{prop-BFI-Sxy-dyadique}. Soient~$x, y$ des réels tels que~$(\log x)^c\leq y\leq x^{1/c}$. On suppose dans un premier temps~$\max\{|a_1|, |a_2|\}\leq x^\delta$. Avec~$I := \floor{3\ee\log x/\log 2}$, on a
\begin{align*} \sum_{r\leq x^{3/5-6\ee}\\(r, a_1a_2)=1}\Big|\sum_{n\in S(x, y)}E_\ee(n;a_1\bar{a_2}, r)\Big|
\leq&\ \sum_{i=0}^{I-1}\sum_{r\leq x^{3/5-6\ee}\\(r, a_1a_2)=1}\Big|\sum_{x2^{-i-1}<n\leq x2^{-i}\\P^+(n)\leq y}E_\ee(n;a_1\bar{a_2}, r)\Big|
\\&\ + \sum_{r\leq x^{3/5-6\ee}\\(r, a_1a_2)=1}\Big|\sum_{n\leq x2^{-I}\\P^+(n)\leq y}E_\ee(n;a_1\bar{a_2}, r)\Big| .\end{align*}
On vérifie que l'on a~$x^{3/5-6\ee}\leq (x2^{-I})^{3/5-4\ee}$, ce qui assure que l'estimation~\eqref{dyad-estim}
de la Proposition~\ref{prop-BFI-Sxy-dyadique} s'applique à chaque sommant de la somme sur~$i$. Par ailleurs,
\[ \sum_{r\leq x^{3/5-6\ee}\\(r, a_1a_2)=1}\Big|\sum_{n\leq x2^{-I}\\P^+(n)\leq y}E_\ee(n;a_1\bar{a_2}, r)\Big|
\ll x^{3/5-6\ee}+\sum_{n\leq x2^{-I}\\a_2n\neq a_1}\tau(|a_2n-a_1|)+x^{2\ee}(\log R)\sum_{n\leq x2^{-I}}1 \ll x^{1-\ee/2} .\]
On obtient donc pour deux réels strictement positifs~$c, \delta$, lorsque~$(\log x)^c\leq y\leq x^{1/c}$,
\begin{equation}\label{majo-E-complete} \sum_{r\leq x^{3/5-6\ee}\\(r, a_1a_2)=1}\Big|\sum_{n\in S(x, y)}E_\ee(n;a_1\bar{a_2}, r)\Big|
\ll (\log x)x^{1-\delta} + x^{1-\ee/2} \ll x^{1-\delta/2} \end{equation}
quitte à réduire la valeur de~$\delta$. En écrivant pour tout~$a\mod{q}$ avec $(a, q)=1$ :
\[ E(x, y; a, q) = \sum_{n\in S(x, y)\\(n, q)=1}E_\ee(n;a, q) + \frac1{\vphi(q)}\sum_{n\in S(x, y)\\(n, q)=1}
\sum_{\chi\prim\\1<\cond(\chi)\leq x^\ee\\\cond(\chi)|q}\chi(n)\bar{\chi(a)}, \]
on obtient l'inégalité suivante, où la variable~$q$ joue le rôle de la variable~$r$ de la majoration~\eqref{majo-E-complete},
\begin{align*}
&\ \sum_{q\leq x^{3/5-6\ee}\\(q, a_1a_2)=1}\Big|E(x, y; a_1\bar{a_2}, q)| \\ \leq&\ 
\sum_{q\leq x^{3/5-6\ee}\\(q, a_1a_2)=1}\Big|\sum_{n\in S(x, y)}E_\ee(n;a_1\bar{a_2}, q)\Big|
+ \sum_{q\leq x^{3/5-6\ee}\\(q, a_1a_2)=1}\frac1{\vphi(q)}\Big|\sum_{n\in S(x, y)\\(n, q)=1}
\sum_{\chi\prim\\1<\cond(\chi)\leq x^\ee\\\cond(\chi)|q}\chi(n\bar{a_1}a_2)\Big| \\
\leq&\ \sum_{q\leq x^{3/5-6\ee}\\(q, a_1a_2)=1}\Big|\sum_{n\in S(x, y)}E_\ee(n;a_1\bar{a_2}, q)\Big|
+ \sum_{q\leq x^{3/5-6\ee}}\frac1{\vphi(q)}\sum_{\chi\mod{q}\\\chi\neq\chi_0\\\cond(\chi)\leq x^\ee}\Big|\sum_{n\in S(x, y)}\chi(n)\Big|.
\end{align*}
Le Lemme~\ref{lemme-harper} implique
\[ \sum_{q\leq x^{3/5-6\ee}}\frac1{\vphi(q)}\sum_{\chi\mod{q}\\\chi\neq\chi_0\\\cond(\chi)\leq x^\ee}\Big|\sum_{n\in S(x, y)}\chi(n)\Big| 
\ll_A \Psi(x, y)\Big\{H(u)^{-\delta}(\log x)^{-A} + y^{-\delta} \Big\} \]
pour tout~$A\geq 0$ (la constante étant effective si~$A<1$), quitte à diminuer la valeur de~$\delta$ et augmenter celle de~$c$. L'inégalité~$x^{1-\delta/2}\ll \Psi(x, y) y^{-\delta/4}$ permet de conclure :
lorsque~$(\log x)^c\leq y\leq x^{1/c}$, on a
\[ \sum_{q\leq x^{3/5-6\ee}\\(q, a_1a_2)=1}|E(x, y; a_1\bar{a_2}, q)| \ll_A \Psi(x, y)\Big\{H(u)^{-\delta}(\log x)^{-A} + y^{-\delta/4} \Big\} .\]

Lorsque~$|a_1|\leq x^{1-\ee}$ et~$|a_2|\leq x^\delta$, on montre par une méthode identique, mais en utilisant l'estimation~\eqref{dyad-estim-unif}
de la Proposition~\ref{prop-BFI-Sxy-dyadique}, que
\[ \sum_{q\leq x^{6/11-7\ee}\\(q, a_1a_2)=1}|E(x, y; a_1\bar{a_2}, q)| \ll_A \Psi(x, y)\Big\{H(u)^{-\delta}(\log x)^{-A} + y^{-\delta/4} \Big\} .\]

\end{proof}

\begin{proof}[Démonstration du Corollaire~\ref{coro-max}]
Soient~$\ee$ fixé et~$c, \delta$ les constantes données par le Théorème~\ref{thm-BFI-Sxy}. On suppose sans perte de généralité que~$\ee<1/5$. On pose
\[ \Delta := 1 + H(u)^{-\delta/2}(\log x)^{-A/2}+y^{-\delta/2}, \quad J := \floor{\frac{\log x}{\log \Delta}\times\frac{\ee}{1+\ee}}, \]
et~$z_j := x\Delta^{-j}$ pour~$j\in\bfN$. Ce choix implique~$\Delta^J\leq x^{\ee/(1+\ee)}$. Lorsque~$z\in[z_j, z_{j+1}]$, ~$y\leq x$ et~$(a, q)=1$, on a
\[ \Psi(z_{j+1}, y ; q, a) - \frac{\Psi_q(z, y)}{\vphi(q)} \leq E(z, y ; q, a) \leq \Psi(z_{j}, y ; q, a) - \frac{\Psi_q(z, y)}{\vphi(q)} \]
ce qui implique
\[ \max_{z\leq x}|E(z, y ; q, a)| \leq \max_{j\geq 0}|E(z_j, y ; q, a)| + \max_{j\geq 0}\frac{\Psi_q(z_j, y) - \Psi_q(z_{j+1}, y)}{\vphi(q)}.\]
Ainsi, lorsque~$2\leq y\leq x$, $Q\leq x$ et~$a_1, a_2\in\bfZ\smallsetminus\{0\}$ tels que~$(a_1, a_2)=1$, on a
\begin{align*}
& \sum_{q\leq Q\\(q, a_1a_2)=1}\max_{z\leq x}|E(z, y ; q, a_1\bar{a_2})| \\  \ll&\ z_J\log x + \sum_{0\leq j\leq J}\Bigg(\sum_{q\leq Q\\(q, a_1a_2)=1}|E(z_j, y ; q, a_1\bar{a_2})|\Bigg) + \sum_{q\leq Q}\max_{j\geq 0}\frac{\Psi_q(z_j, y) - \Psi_q(z_{j+1}, y)}{\vphi(q)} .\end{align*}
Puisque~$x^{3/5-2\ee}\leq z_J^{3/5-\ee}$ et~$x^{6/11-2\ee}\leq z_J^{6/11-\ee}$, le Théorème~\ref{thm-BFI-Sxy} fournit la majoration
\[ \sum_{q\leq Q\\(q, a_1a_2)=1}\big|E(z_j, y ; q, a_1\bar{a_2})\big| \ll \Psi(x, y) \big\{H(u)^{-\delta}(\log x)^{-A} + y^{-\delta} \big\} \]
lorsque~$0\leq j\leq J$, sous les conditions
\begin{equation}(\log x)^c\leq y\leq z_J^{1/c}\quad \text{et}\quad \label{cond-coro}
\begin{cases}Q=x^{3/5-2\ee}, \quad \max\{|a_1|, |a_2|\}\leq x^{\delta/2}, & \text{ou} \\ Q=x^{6/11-2\ee}, \quad |a_1|\leq x^{1-2\ee}, |a_2|\leq x^{\delta/2}.\end{cases} \end{equation}
D'autre part, on a
\[ J \ll (\log x)\Big\{H(u)^{\delta/2}(\log x)^{A/2} + y^{\delta/2}\Big\}, \qquad z_J \ll x^{1-\ee/3}, \]
et cela montre que, toujours sous les conditions~\eqref{cond-coro},
\[ \sum_{q\leq Q\\(q, a_1a_2)=1}\sum_{0\leq j\leq J}|E(z_j, y ; q, a_1\bar{a_2})| \ll \Psi(x, y)\big\{H(u)^{-\delta/2}(\log x)^{-A/2+1} + (\log x)y^{-\delta/2}\big\} .\]
Le theorem~4 de~\cite{Hild85} et le point~(i) du Lemme~\ref{bt05} (avec~$d=\Delta^j(1-\Delta^{-1})^{-1}$) impliquent
\[ \Psi_q(z_j, y) - \Psi_q(z_{j+1}, y) \leq \Psi(z_j, y) - \Psi(z_{j+1}, y) \leq \Psi(z_j(1-\Delta^{-1}), y) \ll \big\{\Delta^{-j}(\Delta-1)\big\}^\alpha\Psi(x, y) ,\]
où on rappelle que~$\alpha=\alpha(x, y)$ est défini en~\eqref{def-alpha}. On en déduit
\[ \sum_{q\leq Q}\max_{j\geq 0}\frac{\Psi_q(z_j, y) - \Psi_q(z_{j+1}, y)}{\vphi(q)} \ll \Psi(x, y)\big\{H(u)^{-\delta\alpha/2}(\log x)^{-A\alpha/2+1} + (\log x)y^{-\delta\alpha/2} \big\} .\]
Quitte à supposer~$c$ suffisamment grand pour avoir~$\alpha\geq 2/3$, on obtient pour tout~$A>0$ la majoration
\[ \sum_{q\leq Q\\(q, a_1a_2)=1}\max_{z\leq x}|E(z, y ; q, a_1\bar{a_2})| \ll \Psi(x, y)\big\{H(u)^{-\delta/3}(\log x)^{-A/3+1} +  (\log x) y^{-\delta/3} \big\} \]
sous les conditions~\eqref{cond-coro}. La condition~$y\leq z_J^{1/c}$ est impliquée par~$y\leq x^{1/(2c)}$ ; les valeurs de~$\ee$ et~$A$ étant arbitraires, on obtient la conclusion souhaitée.
\end{proof}

\section{Application au problème des diviseurs de Titchmarsh friable}

Pour énoncer le résultat de cette section, on reprend quelques notations de~\cite{FT90}. On définit
\begin{equation}\label{def-g-h-etc}
\begin{aligned}
A_0 := \prod_p\Big(1+\frac1{p(p-1)}\Big), &\qquad A_1 := \gamma-\sum_p\frac{\log p}{1+p(p-1)}, \\
g(n) := \prod_{p|n}\Big(1-\frac p{1+p(p-1)}\Big), &\qquad h(n) := \sum_{p|n}\frac{p^2\log p}{(p-1)[1+p(p-1)]}, \\
M_0(t) := &\frac{A_0}{t}\sum_{n\leq t}g(n), \\
M_1(t) := 2A_1M_0(t)+\frac{2A_0}{t}&\sum_{n\leq t}g(n)h(n)-\frac1t\int_1^tM_0(v)\dd v, \\
T_i(x, y) := x\int_{0-}^{+\infty}\rho(u&-v)\dd M_i(y^v) \quad (i\in\{0,1\}).
\end{aligned}
\end{equation}
\begin{prop}\label{divtit-friable}
\begin{enumerate}[(i)]
\item Il existe un réel~$c>0$ tel que lorsque~$(\log x)^c\leq y\leq\exp\{\sqrt{\log x\log_2x}\}$, on ait
\begin{equation}\label{tit-estim-bigu}
T(x, y) = C(\alpha)\Psi(x, y)\log x\Big\{1 + O\Big(\frac1u\Big)\Big\}
.\end{equation}
\item Pour la même constante~$c$, il existe~$\delta>0$ tel que pour tous~$\ee>0, A\geq 0$ fixés,
lorsque~$(x,y)$ est dans le domaine~$(H_\ee)$ et~$y\leq x^{1/c}$, on ait
\begin{equation}\label{tit-estim-smallu}\begin{aligned}
T(x, y) = T_0(x, &y)\log x + T_1(x, y) \\ &+ O_{\ee, A}\big(\Psi(x, y)\big\{H(u)^{-\delta}(\log x)^{-A} + \exp\{-(\log y)^{3/5-\ee}\}\big\}\big).
\end{aligned}\end{equation}
\end{enumerate}
\end{prop}

\begin{proof}[Démonstration que la Proposition~\ref{divtit-friable} implique le Théorème~\ref{tit-bfi}]
Soit~$c$ le réel donné par la Proposition~\ref{divtit-friable}. Lorsque~$y\geq x^{1/c}$, le théorème~1 de~\cite{FT90} s'applique
et entraine la validité de~\eqref{estim-tit-ft} dans ce domaine.

Lorsque~$\exp\{\sqrt{\log x\log_2 x}\}\leq y\leq x^{1/c}$, le théorème~2 de~\cite{FT90} fournit
\[ T_0(x, y)\log x + T_1(x, y) = \Psi(x, y)\log x\Big\{1+O\Big(\frac{\log(u+1)}{\log y}\Big)\Big\} .\]
Or~$(\log x)^{-1}+y^{-\delta}\ll\log(u+1)/\log y$, l'estimation~\eqref{estim-tit-ft} découle donc immédiatement du point~(ii)
de la Proposition~\ref{divtit-friable}.

Lorsque~$(\log x)^c\leq y\leq\exp\{\sqrt{\log x\log_2x}\}$, on a~$1/u\ll\log(u+1)/\log y$, le point~(i) de la Proposition~\ref{divtit-friable}
implique donc la validité de l'estimation~\eqref{estim-tit-ft}.
\end{proof}

\begin{proof}[Démonstration de la Proposition~\ref{divtit-friable}]
On a pour tout~$z\geq 1$ et~$m\in \bfN$,
\[ \tau(m) = \sum_{r|m\\r\leq z}1 + \sum_{r|m\\r<m/z}1 .\]
On a donc pour~$2\leq y\leq x$, en intervertissant les sommations,
\begin{align*}
T(x, y) := \sum_{n\in S(x, y)\\n>1}\tau(n-1) &= \sum_{r\leq z}\sum_{n\in S(x, y), n>1\\n\equiv 1\mod{r}}1
+ \sum_{r<(x-1)/z}\sum_{rz<n\leq x\\ P^+(n)\leq y\\n\equiv 1\mod{r}}1 \\
&= \sum_{r\leq z}\Psi(x, y ; 1, r) + \sum_{r\leq x/z}\{\Psi(x, y ; 1, r) - \Psi(rz, y ; 1, r)\}+ O(z)
\end{align*}
On choisit~$z=\sqrt{x}$. Le Corollaire~\ref{coro-max} assure l'existence de~$\delta, c>0$ tels que pour tout~$A\geq 0$,
\[ \sum_{r\leq\sqrt{x}}\Big|\Psi(x, y ; 1, r) - \frac1{\vphi(r)}\Psi_r(x, y)\Big| \ll \Psi(x, y)\{H(u)^{-\delta}(\log x)^{-A} + y^{-\delta}\} \]
\[ \sum_{r\leq\sqrt{x}}\Big|\Psi(r\sqrt{x}, y ; 1, r) - \frac1{\vphi(r)}\Psi_r(r\sqrt{x}, y)\Big| \ll \Psi(x, y)\{H(u)^{-\delta}(\log x)^{-A} + y^{-\delta}\} \]
uniformément lorsque~$(\log x)^c\leq y\leq x^{1/c}$, ce que l'on suppose dorénavant. On a donc
\begin{equation}\label{lien-T-Tt} T(x, y) = \sum_{r\leq \sqrt x} \frac{2\Psi_r(x, y)-\Psi_r(r\sqrt x, y)}{\vphi(r)}
 + O(\Psi(x, y)\{H(u)^{-\delta}(\log x)^{-A}+y^{-\delta}\}) .\end{equation}
On note~$\tilde{T}(x, y)$ le terme principal du membre de droite.

On suppose dans un premier temps~$(\log x)^c\leq y\leq\exp\{\sqrt{\log x\log_2x}\}$.
Pour~$r\leq x$, on a~$\omega(r)\leq \sqrt{y}$ quitte à supposer~$c\geq 2$. On note
\[ r_y := \prod_{p^\nu||r\\p\leq y}p^\nu .\]
Le point~(ii) du Lemme~\ref{bt05} fournit pour tout~$r$,
\[ \Psi_r(x, y) = \Psi\sb{r\sb y}(x, y) = g\sb{r\sb y}(\alpha)\Psi(x, y)\Big\{1 + O\Big(\frac{E\sb{r\sb y}(1+E\sb{r\sb y})}u\Big)\Big\} \]
où~$E_m$ vérifie
\[ E_m(1+E_m) \ll \frac{\exp\Big\{4\log(\omega(m)+2)\frac{\log u}{\log y}\Big\}}{(\log u)^2} \ll \frac{\omega(m)}{(\log u)^2} \]
quitte à supposer~$c$ suffisamment grand. On a alors
\[ \sum_{r\leq\sqrt{x}}\frac{E\sb{r\sb y}(1+E\sb{r\sb y})}{\vphi(r)}\ll\frac1{(\log u)^2}\sum_{r\leq\sqrt{x}}\frac{\omega(r_y)}{\vphi(r)}
\ll \frac1{(\log u)^2}\sum_{p\leq y}\frac1{p-1}\sum_{r\leq\sqrt{x}/p}\frac1{\vphi(r)} \ll \frac{\log_2y\log x}{(\log u)^2}
\ll \log x .\]
Par ailleurs, le point~(i) du Lemme~\ref{bt05} et une intégration par parties fournissent
\[ \sum_{r\leq\sqrt{x}}\frac{\Psi_r(r\sqrt{x},y)}{\vphi(r)} \ll \Psi(x,y)x^{-\alpha/2}\sum_{r\leq\sqrt{x}}\frac{r^{\alpha}}{\vphi(r)}
\ll \Psi(x, y) .\]
Ainsi,
\[ \tilde T(x,y)=2\Psi(x,y)\sum_{r\leq\sqrt{x}}\frac{g\sb{r\sb y}(\alpha)}{\vphi(r)}+O(\Psi(x,y)\log y) .\]
Une analyse classique, similaire par exemple à la démonstration du Lemme~3.1 de~\cite{FT90}, fournit uniformément pour~$(\log x)^2\leq y\leq x$
\begin{align*}
\sum_{r\leq \sqrt{x}}\frac{g\sb{r\sb y}(\alpha)}{\vphi(r)} &= \prod_{p\leq y}\Big(1-\frac{p^{-\alpha}-p^{-1}}{p-1}\Big)
\prod_{p>y}\Big(1+\frac1{p(p-1)}\Big)\log\sqrt{x} + O(1) \\
&= C(\alpha)\log\sqrt{x} + O(1)
\end{align*}
On a donc~$\tilde T(x, y) = C(\alpha)\Psi(x,y)\log x + O\big(\Psi(x, y)\log y\big)$
et en injectant cela dans~\eqref{lien-T-Tt} on obtient l'estimation~\eqref{tit-estim-bigu}.

On suppose maintenant~$(x, y)\in(H_\ee)$. Le point~(iii) du Lemme~\ref{bt05} fournit
\[ \Psi_r(x, y) = \Psi\sb{r\sb y}(x, y) = \Lambda\sb{r\sb y}(x, y) + O(\Psi(x, y) \exp\{-(\log y)^{3/5-\ee/3}\}(g\sb{r\sb y}(\alpha))^{-1}), \]
on a donc en supposant~$\ee$ suffisamment petit,
\[ \tilde T(x, y) = \sum_{r\leq\sqrt x}\frac{2\Lambda\sb{r\sb y}(x,y)-\Lambda\sb{r\sb y}(r\sqrt x,y)}{\vphi(r)} + O(\Psi(x, y)\exp\{-(\log y)^{3/5-\ee}\}) .\]
En utilisant la définition de~$\Lambda_m(x, y)$, on obtient
\[ \tilde T(x, y) = x\int_{-\infty}^\infty \rho(u-v)\dd W(y^v ; x, y) + O(\Psi(x, y)\exp\{-(\log y)^{3/5-\ee}\}) \]
avec
\[ W(t ; x, y):=\frac1t\sum_{n\leq t}\Big\{2\sum_{r\leq\sqrt x\\(r\sb y,n)=1}\frac1{\vphi(r)}-\sum_{n\sqrt xt^{-1}<r\leq\sqrt x\\(r\sb y,n)=1}\frac1{\vphi(r)}\Big\}.\]
On a~$(r_y, n)=1\Leftrightarrow(r, n_y)=1$, de sorte que le lemme~3.1 de~\cite{FT90} fournit pour~$t\leq x$
\[ W(t ; x, y) = \frac{A_0}t\sum_{n\leq t}g(n_y)\Big\{\log x + 2h(n_y) + 2A_1 - \log\Big(\frac t n\Big)\Big\} + O(x^{-1/3+\ee}) ,\]
où les quantités~$g(n), h(n), A_0$ et~$A_1$ sont définies par~\eqref{def-g-h-etc}.
Pour tout~$z\in\bfC$ avec~$|z|\leq 1$, on pose~$f_z(n) := g(n)\exp\{zh(n)\}$ qui est une fonction entière de~$z$. Pour tout~$n\in\bfN$, on a
\[ |f_{z}(n)| \leq \prod_{p|n}\Big(1-\frac{p}{1+p(p-1)}\Big)\exp\Big(\frac{p^2\log p}{(p-1)[1+p(p-1)]}\Big)
\ll \prod_{p|n}\Big(1+\frac{\log p}{p}\Big) = O((\omega(n)+1)\log(\omega(n)+2)) .\]
On rappelle l'inégalité~$\omega(n)\ll\log (n+1)/\log_2(n+2)$ ({\it cf.}~\cite[théorème~I.5.3]{Tene2007}).

Soient~$F_{z, y}(s)$ et~$F_z(s)$ les séries de Dirichlet associées respectivement aux fonctions multiplicatives~$n\mapsto f_{z}(n_y)$ et~$n\mapsto f_z(n)$.
Elles sont absolument convergentes pour~$\Re(s)>1$.
Avec~$\kappa := 1+1/\log x$, une formule de Perron effective (par exemple~\cite[corollaire~II.2.4]{Tene2007}) fournit pour~$t\leq x$,
\[ \sum_{n\leq t}f_z(n_y) = \frac1{2\pi i}\int_{\kappa-ix^2}^{\kappa+ix^2}F_{z,y}(s)\frac{t^s\dd s}s + O(1) \]
\[ \sum_{n\leq t}f_z(n) = \frac1{2\pi i}\int_{\kappa-ix^2}^{\kappa+ix^2}F_z(s)\frac{t^s\dd s}s + O(1). \]
On vérifie que l'on a uniformément pour~$\sigma>1$,
\[ F_{z,y}(s)=F_z(s)\prod_{p>y}\Big(1+O\Big(\frac{\log p}{p^{\sigma+1}}+\frac1{p^{2\sigma}}\Big)\Big) .\]
Le théorème des nombres premiers fournit donc
\[ F_{z,y}(s) = F_z(s)\Big\{1+O\Big(\frac1{y^{2\sigma-1}}\Big)\Big\} \]
et on obtient pour~$1\leq t\leq x$ l'estimation
\[ \sum_{n\leq t}f_z(n_y) = \sum_{n\leq t}f_z(n) + O\Big(\frac{t}{y}\int_{-x^2}^{x^2}\sum_{n\geq 1}\frac{|f_z(n)|}{n^\kappa}\frac{\dd\tau}{1+|\tau|}+1\Big) .\]
Lorsque~$t\leq y$, les deux termes principaux sont égaux : le terme d'erreur est donc~$O(t(\log x)^3/y)$.
De même que dans~\cite{FT90}, le cas~$z=0$ fournit directement
\begin{equation}\label{approx-g} \frac1t\sum_{n\leq t}g(n_y) = \frac1t\sum_{n\leq t}g(n) + O\Big(\frac{(\log x)^3}{y}\Big)
\qquad(t\leq x),\end{equation}
tandis que les formules de Cauchy
\[ \sum_{n\leq t}g(n_y)h(n_y) = \frac1{2\pi i}\oint_{|z|=1}\sum_{n\leq t}f_z(n_y)\frac{\dd z}{z^2}, \qquad
\sum_{n\leq t}g(n)h(n) = \frac1{2\pi i}\oint_{|z|=1}\sum_{n\leq t}f_{z}(n)\frac{\dd z}{z^2} \]
fournissent
\begin{equation}\label{approx-gh} \frac1t\sum_{n\leq t}g(n_y)h(n_y) = \frac1t\sum_{n\leq t}g(n)h(n) + O\Big(\frac{(\log x)^3}{y}\Big) 
\qquad(t\leq x).\end{equation}
Enfin, par une intégration par parties à partir de la formule~\eqref{approx-g}, on a
\begin{equation}\label{approx-glog} \frac1t\sum_{n\leq t}g(n_y)\log\big(\frac t n\big) = \frac1t\sum_{n\leq t}g(n)\log\big(\frac t n\big)
+ O\Big(\frac{(\log x)^4}{y}\Big) \qquad(t\leq x).\end{equation}
Les formules~\eqref{approx-g}, \eqref{approx-gh} et~\eqref{approx-glog} fournissent pour~$t\leq x$
\begin{align*}
W(t ; x, y) &= \frac{A_0}t\sum_{n\leq t}g(n)\Big\{\log x + 2h(n) + 2A_1 - \log\Big(\frac t n\Big)\Big\} + O(x^{-1/3+\ee} + (\log x)^4y^{-1}) \\
&= M_0(t)\log x + M_1(t) + O(x^{-1/3+\ee} + (\log x)^4y^{-1})
,\end{align*}
de sorte que l'on obtient
\[ \tilde T(x, y) = T_0(x, y)\log x + T_1(x, y) + O(\Psi(x, y)\exp\{-(\log y)^{3/5-\ee}\}) \]
puisque~$x^{-1/3+\ee} + (\log x)^4y^{-1} \ll \exp\{-(\log y)^{3/5-\ee}\}$.
En injectant cela dans~\eqref{lien-T-Tt}, on obtient le point~(ii) de la Proposition~\ref{divtit-friable}.
\end{proof}

\bibliographystyle{amsalpha}
\bibliography{bffi-friable-v3}

\end{document}